\documentclass[11pt,fullpage]{article}

\title{
	Hamiltonian and Pseudo-Hamiltonian Cycles and Fillings In Simplicial Complexes 
}

\usepackage{authblk}

\author{Rogers~Mathew, $~ $ Ilan~Newman, $~ $ Yuri~Rabinovich and 
 $~ $ Deepak~Rajendraprasad}

\usepackage{amssymb,fullpage}
\usepackage{xspace}
\usepackage{latexsym}
\usepackage{times}
\usepackage{amsfonts}
\usepackage{amsmath}
\usepackage{color}
\definecolor{Blue}{rgb}{0.3,0.3,0.9}
\definecolor{Black}{rgb}{0.0,0.0,0.0}
\usepackage{graphicx}
\usepackage{epstopdf}
\usepackage{titling}

\usepackage[margin=2cm]{geometry}
\newcommand{\ignore}[1]{}
\def \qed {\hspace*{0pt} \hfill {\quad \vrule height 1ex width 1ex depth 0pt}
 \medskip}

\newenvironment{proof}{\par\noindent{\bf Proof.}\quad}{  $\qed$ \newline}

\newtheorem{theorem}{Theorem}[section]

\newtheorem{corollary}[theorem]{Corollary}
\newtheorem{conjecture}[theorem]{Conjecture}
\newtheorem{lemma}[theorem]{Lemma}
\newtheorem{claim}[theorem]{Claim}

\newtheorem{definition}[theorem]{Definition}

\def \deficit {{\rm deficit}}
\def \xor {\oplus}
\def\into{\longrightarrow}
\def\N{\mathbb{N}}

\def\Q{\mathbb{Q}}
\def\R{\mathbb{R}}
\def\F{\mathbb{F}}

\def\Link{Lk}
\def\Star{St}
\def\Cone{Cone}

\def\Fill{{\rm Fill}}
\def\FILL{{\rm FILL}}

\def\supp{supp}
\def\boundary{\partial}

\def\C{\mathcal{C}}
\makeatletter

\newcommand{\Rmnum}[1]{\expandafter\@slowromancap\romannumeral #1@}
\makeatother

\begin{document}

\maketitle

\begin{abstract}
We introduce and study a $d$-dimensional generalization of Hamiltonian
cycles in graphs -  the Hamiltonian $d$-cycles in $K_n^d$ (the complete
simplicial $d$-complex over a vertex set of size $n$).
Those are the simple $d$-cycles of a complete rank, or, equivalently, of
size $1 + {{n-1} \choose d}$. 

The discussion is restricted to the fields $\F_2$ and $\Q$.
For $d=2$, we characterize the $n$'s for which Hamiltonian $2$-cycles exist.
For $d=3$ it is shown that Hamiltonian $3$-cycles exist for infinitely many $n$'s. 
In general, it is shown that there always exist simple $d$-cycles of size
${{n-1} \choose d} - O(n^{d-3})$. All the above results are constructive.

Our approach naturally extends to (and in fact, involves) $d$-fillings, generalizing the
notion of $T$-joins in graphs. Given a $(d-1)$-cycle $Z^{d-1} \in K_n^d$,
~$F$ is its $d$-filling if $\partial F =   Z^{d-1}$.  We call a $d$-filling
Hamiltonian if it is acyclic and of a complete rank, or, equivalently, is  of
size ${{n-1} \choose d}$. 
If a Hamiltonian $d$-cycle $Z$ over $\F_2$  
contains  a $d$-simplex $\sigma$, then $Z\setminus \sigma$ is a a
Hamiltonian $d$-filling of $\partial \sigma$ (a closely related fact  is also true for
cycles over $\Q$). Thus, the two 
notions are closely related.

Most of the above results about Hamiltonian $d$-cycles hold for 
Hamiltonian $d$-fillings as well.
\end{abstract}

{\bf Categories:  math.CO, math.AT}
\section{Introduction}
Combinatorial topology (more precisely, the Homology theory for simplicial complexes)  provides a natural framework 
allowing to generalize the fundamental graph-theoretic 
notions such as cycles, trees, cuts, expanders, Laplacians, etc., to $(d+1)$-uniform hypergraphs, viewed as pure 
$d$-dimensional simplicial complexes. 
Historically, this framework was used and developed mostly to serve the needs of other disciplines, 
first and foremost the Algebraic Topology, and, more recently, e.g., the digital processing of visual data. In recent decades it came under investigation for its own sake, resulting in new beautiful results and applications, see~\cite{linial-peled, LuMe, MSTW, 
Golubev} to name but a few.  

The key notions studied in this paper are  $d$-cycles and
acyclic $d$-fillings of a maximum possible size. For simplicity consider first the one-dimensional case over the field $\F_2$, i.e., graphs. Given a set $E$ of edges over the vertex set $V$, define $\partial_1 E $, the boundary of $E$, as the set of all
vertices incident to an odd number of edges in $E$.
The set $E$ is  a $1$-cycle if $\partial_1(E) =\emptyset$. A set
$E$ is called acyclic if it contains no cycles. A maximal
acyclic set is called a $1$-tree. It is a basic fact that all maximal acyclic sets have the same size, which is $|V|-1$. It is a simple exercise to show that for any even set of vertices$Z \subset V$ \footnote{This is a
$0$-dim cycle, see Section \ref{sectionSimplicialComplexes}.}, there exists a set of edges 
$F$ over $V$ with $\partial_1 F=Z$. I.e., it is a graph whose set of odd degree
vertices is $Z$. Such $F$ is classically called a $Z$-join. 
In view of higher-dimensional generalization to come, we shall call it
 a $1$-filling of $Z$. It is easy to verify that there exists a 1-fillings of $Z$
 that is acyclic. In a special case when $Z = \{a,b\}$,  an acyclic filling of $Z$ of the largest possible size is a Hamiltonian path
whose end points are $Z$. Together with the pair $(a,b)$ it forms
a Hamiltonian cycle - the largest possible simple
cycle (that is, as cycle that does not contain a proper cycle as a subset).

This can naturally be generalized to higher dimensions: instead of
pairs, let $T$ be a set of triplets (sets of size $3$) over the set
of vertices $V$. In this case the boundary $\partial_2(T)$ (over $\F_2$) is the set
of pairs of vertices, each that is incident to an odd number of
triangles. Again, a $2$-cycle is a set of triplets with empty boundary, and
acyclic sets of triplets are those containing no cycles. A simple
cycle is a cycle that does not contain a proper subset that is by itself a cycle. It turns
out (from the same algebraic reasoning as for graphs) that all the maximal acyclic sets
 have the same size, which is ${n-1 \choose 2}$, where $n= |V|$. 
In addition,   any $1$-cycle $Z$ over $V$ has an acyclic $2$-filling
$F$, i.e., an acyclic set of triplets $F$ with $\partial_2 F =
Z$. 

How large can a simple $2$-cycle over an $n$-size vertex set be?  The two-dimensional case is much less
obvious, and to our best knowledge, was not systematically studied so
far.
The following upper bound is simple: the removal of a triangle form a simple
cycle creates an acyclic set. Since all acyclic sets are of size at
most  $r(n,2)={{n-1} \choose 2}$, it follows that an absolute upper
bound is $r(n,2)+1$.  Would such a simple $2$-cycle exist, it would  be
called Hamiltonian $2$-cycle. Note the connection to fillings:  If
$Z$ is a simple $2$-cycle containing a triangle $\sigma$, then $Z
\setminus \{\sigma\}$ is a $2$-filling of the three pairs that are the
boundary of $\sigma$ (and are a $1$-cycle).

 For the lower bound on the largest $2$-simple cycle, is has been known for some time that there exist simple $2$-cycles of size $c_2 \cdot r(n,2)$ for some constant $0<c_2<1$. E.g., the important Complete Graph Embedding Theorem 
(implying the tightness of Heawood's bounds on the chromatic number of graphs embeddable
in 2-surfaces of a prescribed genus; see e.g., the book~\cite{Gross-Tucker}) 
claims that any $K_n$, $n\geq 4, \; n\equiv 0, 1 \, ({\rm mod}\, 3)$, is (efficiently) realizable as a triangulation of (both orientable, and nonorientable) 2-surface. This gives an explicit construction of a simple $2$-cycle of a size $\approx {2\over 3} r(n,2)$. 

All the above notions are generalized to higher dimensions. In this
case the size of a maximum simple $d$-cycle on $V$ of size n is at most $r(n,d)+1$, where $r(n,d)= 
{{n-1} \choose d}$, due to the rank argument, and at least $c_d \cdot r(n,d)$ for some (small) constant $c_d>0$.
This follows, e.g., from the study of the threshold probabilities for
random simplicial $d$-complexes by Linial et
al.~\cite{ALiLuMe}. 

It this paper we completely resolve the two-dimensional case, and
(constructively) show that the size of a largest simple $2$-cycle is $r(n,2)$ when
$n\equiv 1, 2 \, ({\rm mod} \, 4)$, and 
$r(n,2)+1$ when $n\equiv 0, 3 \, ({\rm mod}\, 4)$. Hence, 
Hamiltonian 2-cycles exist for the latter case.  In dimension 3 we construct
Hamiltonian simple $3$-cycles, that is of size $r(n,3)+1$, for an
infinite sequence of $n$'s and in general, 
we construct simple $d$-dimensional cycles of size
$(1- O(1/n^3))\cdot r(n,d)$.

Observe that any nontrivial simple $d$-cycle $Z$ can be represented
in a form $Z=\sigma_d - F^{(d)}$, where $\sigma_d \in Z$ is a $d$-simplex,
and $F^{(d)}$ is an acyclic $d$-filling of $\partial_d \sigma_d$. Thus,
constructing large simple $d$-cycles is equivalent to constructing
large acyclic $d$-fillings of $\partial_d\sigma_d$. It is natural to
generalize this question to what is the maximum possible size of
an acyclic $d$-filling $F^{(d)}$ of a (any) given nontrivial $(d-1)$-cycle
$Z$, with respect to set of vertices $V$, $|V|=n$. The rank
argument immediately implies that $|F^{(d)}| \leq r(n,d)$.  
  For $d=2$ we completely resolve the case and for $d>2$ we construct an acyclic $d$-filling of size
$(1- O(1/n^3))\cdot r(n,d)$ for any nontrivial $(d-1)$-cycle
$Z$.

We end with a remark that while the basic definition of boundary was
defined above with respect to $\F_2$, all notions and results extend also to
boundaries with respect to $\Q$, or any other field. 
\ignore{
Another related extremal problem studied in this paper is the following:
 Consider first the one-dimensional case. 
 The fundamental cycle of an edge $e$ with respect to a spanning tree
 $T$ of a complete graph $K_n,$ is the (unique) cycle created by
 adding $e$ to $T$. Given an edge $e$, we know that there exist $T$'s
 for which the fundamental cycle of $e$ w.r.t. $T$ is of size
 $r(n,1)+1=n$ (which is maximim possible). These are precisely the
 Hamiltonian paths. What about $T$'s whose {\em average} fundamental
 cycle (with respect to the random uniform distribution of edges) is
 as large as possible? A moment's reflection shows that those are
 again the Hamiltonian paths, with the average size of a fundamental
 cycle $\approx {1 \over 3} n$.

The above notions and  the problem itself, are easily lifted to higher dimensions.
The question becomes: How large can be the average fundamental $d$-cycle of a spanning $d$-hypertree on 
a vertex set $V$, $|V|=n$, where the averaging is done over all $d$-simplices $\sigma$ on $V$?
We show that this quantity is at least $c_d\cdot r(n,d)$, where $c_d
\approx 4^{-d}$.
}
 
Finally, a note about the methods: The paper is combinatorial in
nature. Its use of Homology theory does not go beyond the basic definitions, and the
basic properties of the resulting structures. This is partially due to a systematic use of a very special type of acyclic sets of 
$d$-simplices, and the $d$-chains supported on them. 
Such sets, defined in a purely combinatorial manner by means of a certain conical extension (see Claim~\ref{cl:cone_extension} below), are quite tractable by combinatorial means, and may prove useful for future studies.  
\subsection{Terminology and Preliminaries Pertaining to Simplicial Complexes}
\label{sectionSimplicialComplexes}
\subsubsection{Basic Standard Notations}
The notation $[n]$ is a shorthand for the set $\{1, \ldots, n\}$. 
If $A$ and $B$ are sets, then $A \oplus B$ denotes their symmetric difference; if $A$ and $B$ are vectors over $\F_2$, then 
it denotes their vector sum. 

{\bf simplices and Complexes.} An abstract {\em $d$-dimensional
  simplex} (or {\em $d$-simplex} for short) can be identified with a
set of size  $d+1$.  An abstract {\em simplicial complex} $X$ is
a collection of simplices that is closed under containment. In this
case, the simplices in $X$ are also called {\em faces}.
The set of all the $0$-simplices in $X$ is called the {\em vertex-set} $V(X)$ of $X$. In this paper we shall always assume 
that $V(X)$ is finite and often identify it with $[n]$, where $n = |V(X)|$. 
The {\em dimension} of a simplex is the size of its vertex set minus $1$. 
The {\em dimension} of a simplicial complex $X$ is the maximum dimension of a simplex in $X$.
 Further, $X$ is called {\em pure} if all its maximal faces are of the same dimension.

The set of all $i$-dimensional simplices of $X$, the {\em $i$-skeleton} of $X$, is denoted by $X^{(i)}$. 

The {\em complete} $d$-dimensional simplicial complex
on $[n]$,  $K^d_n = \{ \sigma \subset [n] : |\sigma| \leq d+1\}$, contains all the simplices on $[n]$ of dimension $\leq d$.

The {\em degree} of a $k$-face $\sigma$ in a {\em pure} $d$-dim simplicial complex $X$,
denoted $deg(\sigma,X)$, is the number of $d$-faces in $X$ which
contain $\sigma$.

{\bf Orientations, Chains, and the Boundary Operator.}
An {\em orientation} of a simplex is the equivalence relation on all
the permutations on $V(\sigma)$, that is - orderings of the vertices,
in which two permutations are equivalent if one being an even
permutation of the other. Hence, there are two possible orientations of a $d$-simplex of dimension $\geq 2$, and one orientation for $d<2$. An {\em oriented simplex} is a simplex with orientation. 
An oriented simplicial complex is a simplicial complex whose simplices are oriented. 

Given a field $\F$ and an oriented simplicial complex $X$, an $\F$-weighted formal sum $C$ of the 
(oriented) $k$-faces of $X$ is called a {\em $k$-chain} on $X$ over
$\F$, i.e., $C = \sum_{\sigma \in X^{(k)}}{c_{\sigma} \sigma}$, where
$c_{\sigma} \in \F$. All different orderings of a $d$-simplex are
divided to two equivalent classes, represented by the $\{-1,+1\}$
signs. Over $\F_2$ the notion of a sign is vacuous. The importance of
the signs is when considering the boundary operator, to be discussed below.

The support $\supp(C)$ of $k$-chain $C$ is the set of non-oriented
$k$-simplices $\sigma$ such that $c_\sigma \neq 0$. The {\em size} of
$C$ is defined as $|C|=|\supp(C)|$.  The collection of all $d$-chains
on $K^d_n$ form a vector space $\C_d$ of dimension ${n \choose d+1}$.
The vertex set of a chain $C$ over $K_n^d$ is $V(C) = V(\supp(C))$.

The {\em boundary} $\boundary_d \sigma$ of an oriented $d$-simplex $\sigma = \{v_0, \ldots, v_d\}$, with $v_0 < \cdots < v_d$, is the $(d-1)$-chain $\sum_{i=0}^d (-1)^i \sigma_i$,
where $\sigma_i = (\sigma \setminus \{v_i\})$ is the oriented simplex obtained by erasing $v_i$ from the oriented $\sigma$ as above. The boundary operator is well defined in the sense that it does not depend on the particular orderings (up to corresponding equivalences) chosen to represent 
$\sigma$ and $\sigma_i$'s respectively. Note that $\tau =\sigma - \{v
\}$ has a sign above depending on the relative order of $v$ in
$\sigma$. We denote this sign by $[\sigma: \tau]$. Hence $\partial
\sigma = \sum_{v \in \sigma} [\sigma:(\sigma\setminus \{v\}] \cdot \sigma$.  
The linear extension of this operator to the whole of $\C_d$ 
is the {\em boundary operator} $\boundary_d : \C_d \into \C_{d-1}$. 
A fundamental property of the boundary operator is  $\boundary_{d-1} \boundary_d = 0$. 

When the value of $d$ is unambiguous from the context, the subscript $d$ of $\partial_d$ may be dropped. 

{\bf Cycles.} A $d$-chain $Z$ is called a {\em $d$-cycle} if
$\boundary_d Z = 0$. We refer to $0 \in \C_d$  as the {\em trivial
  $d$-cycle} or the {\em zero cycle}. Further, when $Z$ is the only nontrivial $d$-cycle supported on a $\supp(Z)$, $Z$ is called {\em simple}. 
The collection of all $d$-cycles of $K^d_n$ form a vector space ${\cal
  Z}_d$ of dimension ${{n-1} \choose d+1}$ over $\F$. Note that for
$(d+1)$-simplex $\sigma,$ $\partial \sigma$ is a non-trivial $d$-cycle. This
is the non-trivial cycle of minimum possible size (for any dimension). It
can be verified that the space of $d$-cycles ${\cal
  Z}_d$  is spanned by
$\{\partial_{d+1} \sigma: ~ \sigma \in K_n^{d+1} \}$.

{\bf Forests and Hypertrees.}
A pure $d$-complex $F$ is called {\em acyclic} if there no nontrivial $d$-cycle
whose support is a subset of $\supp(F)$. Slightly deviating from the standard notation, we shall
call such set of $d$-simplices $F$ a {\em $d$-forest}, and, further, call it a {\em $d$-hypertree} on $[n]$ if it is a maximal $d$-forest in $K^d_n$. Matroid-theoretic considerations immediately imply
that all $d$-hypertrees on $[n]$ have the same size. Consider the $d$-star in $K^d_n$, i.e., the set
of all $d$-simplices that contain a fixed vertex $v$. One can easy verify that it is a maximal forest, i.e.,
a $d$-hypertree. Hence, the size of any $d$-hypertree of $K^d_n$ is equal to the size of $d$-star,
being ${n-1 \choose d}$. 

The set of all $(d-1)$-chains $\{\boundary_d \sigma : \sigma \in F\} \subset {\cal Z}_{d-1}$
 is linearly independent when $F$ is a $d$-forest, and,
 moreover, it is a basis of ${\cal Z}_{d-1}$ when
 $F$ is a $d$-hypertree. This is the {\em spanning property} of $d$-hypertrees. In particular, for such
 $F$ every 
$d$-simplex $\sigma \in K^d_n \setminus \supp(F)$ defines the {\em fundamental $d$-cycle of $\sigma$ with respect to $T$}, 
being the support of the unique non-trivial $d$-cycle supported on the $F \cup \{\sigma\}$. 

\ignore{
{\bf Collapsibility.}
A $(d-1)$-face $\tau$ of a pure $d$-dimensional simplicial complex $X$
is called {\em exposed} if its degree is $1$, that is, it belongs to
exactly one $d$-face $\sigma$ of $X$. An elementary $d$-collapse on
$\tau$ consists of the removal of $\sigma$ and $\tau$ from $X$. The
complex $X$ is {\em collapsible to its $(d-1)$-skeleton} if every
$d$-face of $X$ can be removed by a sequence of elementary collapses
of $(d-1)$-facets. It is easy to see that such $X$ is acyclic over any
field, i.e., $X^{(d)}$ is always a $d$-forest. For $d=1$ (that is, $X$
is a graph), being a forest is equivalent to be collapsible. In terms
of graph theory, this is equivalent to the fact that every non-trivial
forest
contains a leaf (a degree one vertex). For higher dimension this is
not the case. 
}


{\bf Hypercuts.} $d$-hypercuts of $K_n^d$ are its $d$-cocycles
(equivalently, $d$-coboundaries) of a minimal support. To avoid the
unnecessary discussion of $d$-cochains and $d$-cocycles, for the needs
of this paper it suffices to say that the supports $S$ of
$d$-hypercuts are precisely the sets of $d$-simplices obtainable in
the following manner. Start with any  $d$-hypertree $T$ of $K_n^d$ and
$\sigma \in T$. Then, set $S$ to be the set of all $d$-simplices
$\tau$ such that $T\setminus \{\sigma\} \cup \{\tau\}$ is acyclic. 
See~\cite{NR, LNPR} for more details
on $d$-hypercuts.  

Finally, we note that over $\F_2$, $d$-chains (that is, cycles in this
context) can be identified with their support.

\subsubsection{Less Common Notions, Operators and Facts}\label{sec:less_common}
{\bf Star  and  \bf Link.} While these operators are usually considered in the context 
of simplicial complexes, they are well defined for chains as well.
Given a $d$-simplex $\sigma$ and a vertex $v$ the {\em star} of $\sigma$ with respect to $v$ is 
$\Star(v,\sigma) = 0$ if $v \notin \sigma$ and $\sigma$
otherwise. Similarly $\Link(v,\sigma) = [\sigma:(\sigma \setminus\{v\})] \cdot
(\sigma \setminus \{v\})$.  Both operation are extended linearly to
chains.

Note that $\Link(v,\sigma) = \partial \sigma - \Star(v, \partial
\sigma)$. 

It follows immediately that  a link of a $d$-cycle $Z^d$ is a
$(d-1)$-cycle over $V \setminus \{v\}$  since this is immediate for
the cycle $\partial \sigma$ for any $(d+1)$-simplex $\sigma$, and as
commented above these cycles span the space of cycles.
 
\noindent
{\bf Cone.} The {\em cone} operator is the right inverse of the link operator;
For $x \notin \sigma$ it maps a $d$-simplex $\sigma$ to the
$(d+1)$-simplex $\Cone(x,\sigma) = [(\sigma \cup \{v\})~ : \sigma] \cdot (\sigma \cup
\{v\})$. Again, this is linearly extended to any chain $C$ where $x
\notin V(C)$.

A simple verification yields:
\begin{equation}
\label{eqnConeLink}
	 \Cone(x, \Link(x,C)) ~=~ \Star(x,C)\,.
\end{equation}
and 
\begin{equation}
\label{eqnBoundaryLift}
	 \partial_{d+1}\Cone(x,C) ~=~ C - \Cone(x,\partial_d C)\,. 
\end{equation}

The following fact about conic extensions is fundamental for this paper. Observe that (with some abuse of notation) the $\Cone(x,S)$ operator is well defined not only for $d$-chains, but also for non-oriented unweighted sets of $d$-simplices.
\begin{claim}
\label{cl:cone_extension}
Assume that $T^{(d)}$ and $T^{(d-1)}$ are, respectively, a $d$-forest and a $(d-1)$-forest (a $d$-hypertree and a $(d-1)$-hypertree) over a field $\F$ and a vertex set $V$. Then, for $x \not\in V$, $T^{(d)} \cup \Cone(x,T^{(d-1)})$ is a $d$-forest (a $d$-hypertree) over $V \cup \{x\}$.
\end{claim} 
\begin{proof} Since $T^{(d)}$ is acyclic and disjoint from $\Cone(x,T^{(d-1)})$, 
any nontrivial $d$-cycle $Z$ supported on $T^{(d)} \cup \Cone(x,T^{(d-1)})$ must contain the
the vertex $x$. Consider $\Link(x,Z)$. On one hand it is a nontrivial $(d-1)$-cycle on $V$.
On the other hand, it is supported on the acyclic $T^{(d-1)}$:  Contradiction.

Further, set $|V| = n$. If  $T^{(d)}$ and $T^{(d-1)}$ are {\em hypertrees} over $V$, they
have support of size ${{n-1} \choose d}, ~ {{n-1} \choose {d-1}}$
respectively. Then   $T^{(d)} \cup \Cone(x,T^{(d-1)}) $ has support of size   ${{n-1} \choose d} +
{{n-1} \choose {d-1}} = {n \choose d}$, and therefore a $d$-hypertree
over $V \cup \{x\}$.
\end{proof}
\noindent
{\bf A matter of notations}\\
{\em In what follows we often use a superscript $d$ over a chain or a
simplicial complex. The superscript denotes the
    maximal dimension of the corresponding (usually pure) object. 
$Z$ will always denote a cycle, $F$ or $T$ will denote acyclic chains
or sets (that is, forests). Hence e.g., $Z^d$ is a $d$-cycle.}

\vspace{0.7cm}
\noindent
{\bf Fillings.}
 A {\em filling} of a  $(d-1)$-cycle\footnote{Formally, fillings should be defined for  $(d-1)$-boundaries 
rather than for $(d-1)$-cycles. However, for $K_n^d$, as well as for any homologically $d$-connected complex, 
the two are the same.} $Z^{d-1}$ over $K_n^d$ is a $d$-chain $F^{(d)}$ over $K_n^d$ such that 
$\partial F^{(d)} = Z^{d-1}$. A filling $F^{(d)}$ (and in general, any
$d$-chain) will be called {\em acyclic}  if its support is
acyclic. The fact that  $F$ is a filling of $Z^{d-1}$ will be denoted
as $F = \Fill(Z^{d-1})$.

The {\em deficit} of an acyclic chain $F^{(d)}$ will be defined as
$\deficit  (F_n^d) = {{n-1} \choose d} - |F^{(d)}|$.
Since ${{n-1} \choose d}$ is the size of every maximal  acyclic $d$-chain in $K_n^d$, the deficit
is never negative. 

Let $T\subseteq K_n^d$ be a $d$-hypertree. For every $(d-1)$-cycle
$Z^{d-1}$ on $K_n^d$ there exists a unique acyclic filling of
$Z^{d-1}$ supported on $T$. This immediately follows from the spanning
property and the acyclicity of $T$. In fact, this is a linear
bijection between ${\cal Z}_{d-1}$, the set of $(d-1)$-cycles of
$K_n^d$, and ${\cal C}_d(T)$, the set of $d$-chains supported on a
$T$. 

{\bf $0$-deficit  fillings, Hamiltonicity and cycles.}
A a $0$-deficit acyclic filling $F^{(d)}$ of $Z^{d-1}$ in $K_n^d$ is obviously the
largest possible filling (in terms of its support). If $Z^{d-1}
= \partial \sigma$ for some $\sigma \in K_n^d$,  a $0$-deficit acyclic
filling $F^{(d)}$ of $\partial \sigma$ will be called {\em Hamiltonian}
as $F - \sigma$ is a {\em simple} cycle of the maximum possible support, namely
${{n-1} \choose d} +1$.
In turn,  a simple $d$-cycle $Z^d$ in $K_n^d$ will be called Hamiltonian if its size is
${{n-1} \choose d}+1$. Observe that $Z^d$ is Hamiltonian if and only if for any term $c_\sigma \sigma$
in it (where $\sigma$ is a $d$-simplex), $Z^d -  c_\sigma \sigma$ is an acyclic $0$-deficit filling
of $\partial \sigma$. 

While for graphs Hamiltonian cycles always exist (for any $n\geq 3$), this is not necessarily true for higher dimensional
full-simplicial complexes. 

\section{Large Acyclic $d$-Dimensional Fillings}
\label{sectionFillings}
Can one expect that every $(d-1)$-cycle $Z^{d-1}$ on $K_n^d$ has a
$0$-deficit filling? In particular, is there a Hamiltonian
$d$-cycle for every $d$ for large enough $n$?    
The answer may depend on the underlying field. For $\F_2$ there 
 is an obvious obstacle for fillings of
$(d-1)$-cycles, for even $d$.
Observe that in this case $\partial_{d} F^{(d)} = Z^{d-1}$ implies 
that the sum of coefficients (mod 2) of the chain $F^{(d)}$ is equal to that of $Z^{d-1}$. In other words,
the parities of $|F^{(d)}|$ and  $|Z^{d-1}|$ must be equal. We call this
obstacle 'the parity condition', and it is defined formally below. 

Thus if $Z^{d-1}$ has a $0$-deficit filling 
the following parity condition holds.
\begin{definition}[parity condition]
We say that a non-empty $(d-1)$-cycle over $\F_2$ has the parity
condition if $d$ is even and 
\begin{equation}\label{eq:1}
|Z^{d-1}| \equiv {{n-1} \choose d}~({\rm mod}\;2)  
\end{equation}
\end{definition}


For all we presently know, the following rather strong conjecture may well be true:
\begin{conjecture}
\label{conj:main}
Over $\F_2$, for every $d\geq 0$ there exists a number $n_d$, such
that every non-trivial $(d-1)$-cycle $Z^{d-1}$ on $K_n^d$ with $n \geq
n_d$ has a $0$-deficit filling if an only if the parity condition
holds. 
More over, for any non-trivial $(d-1)$-cycle, regardless of the parity
condition there is an acyclic filling $F^{(d)}$ of $Z^{d-1}$ of deficit $1$.

Over $\Q$, for significantly large $n$, $Z^{d-1}$  always has a
$0$-deficit filling on $K_n^d$. 
\end{conjecture}

In what follows we shall establish this conjecture for $d \leq 2$
(over $\F_2$ and over $\Q$). 

\begin{theorem}
  \label{thm:d2}
Over $\F_2$, every nonzero $1$-cycle $Z^1$ on $K_n^2$ has an acyclic
filling of deficit at most $1$. Further, if   the parity condition
holds it has a $0$-deficit acyclic filling. 

Over $\Q$, every nonzero $1$-cycle $Z^1$   has a
$0$-deficit acyclic filling on $K_n^2$ for   large enough $n$.
\end{theorem}

For $d\geq 3$, we prove a weaker statement:
\begin{theorem}
\label{th:d>3}
Using the notations of Conjecture~\ref{conj:main}, there always exists
an acyclic filling $F^{(d)}$  of $Z^{d-1}$ (over $\F_2$ and over $\Q$) on 
$K_n^d$ of deficit $O(n^{d-3})$. In particular, for $d=3$, the deficit is constant.
\end{theorem}
%
%
%
%

In all cases the following generic recursive construction, $\FILL()$ will be employed. 
Given a nonzero $(d-1)$-cycle $Z_n^{d-1}$ over $K_n^d$ it reduces the
problem to constructing a (large) acyclic $(d-1)$-dimensional fillings
for a certain $(d-2)$-cycle and an acyclic filling of a $(d-1)$-cycle,
but over a smaller underlying set.

\vspace{1cm}
\noindent
{\bf A matter of notations}\\
{\em{ In what follows the universe over which all simplicial complexes
    are considered is $V = [n]$. All chains in what follows are pure
    and are denoted using a
    subscript and a superscript. The superscript denotes the
    maximal dimension while the subscript denotes the size of the
    subset of the universe on which the chain is defined over. The
    the actual subset of vertices will be either clear from the
    context, or explicitly defined.

In the recursion below we initially have our universe
$V=[n]$. However, during the recursive procedure we choose a special
vertex $v_n \in V$. This will define a re-enumeration of $V$ along
every recursion path according to this order in which the vertices are
chosen. Once $v_n$ is chosen, some next objects over $V
\setminus \{v_n\}$  are (recursively) constructed and hence
their subscripts will correspondingly be $(n-1)$.
}
\\ 

\noindent
{\bf 
$\FILL(Z_n^{d-1},V) ~ ~$ } ; the input $Z_n^{d-1}$ is a $(d-1)$-cycle over the
universe $V$.  \\

\vspace{-0.4cm}
\hspace{2.2cm} ; The result is an acyclic filling $F_n^d$
of $Z_n^{d-1}$. 

\noindent
If $Z_n^{d-1}=0$ return $0$ (the zero cycle  $\leftrightarrow$  empty filling).

\noindent
if $d=0$, and $Z_n^{-1} = c\cdot\emptyset$,  return a (suitably chosen) vertex 
$v \in V$ with coefficient $c$; \\

\noindent
if $d>0$, \\
$\mbox{}$~~~~~ pick a (suitably chosen) pivot vertex $v_n$  in $V(Z_n^{d-1})$\;; \\ 
$\mbox{}$~~~~~~~~~ $Z_{n-1}^{d-2} ~\leftarrow~ \Link(v_n, Z_n^{d-1})$\;;\\
$\mbox{}$~~~~~~~~~ $F_{n-1}^{d-1} ~\leftarrow~ \FILL(Z_{n-1}^{d-2},
V\setminus \{v_n\})$\;;\\
$\mbox{}$~~~~~~~~~ $Z_{n-1}^{d-1} ~\leftarrow~ 
              Z_n^{d-1} \; - \; \Star(v_n, Z_n^{d-1})\;+\;F_{n-1}^{d-1} $\;; \\
\noindent
$\mbox{}$~~~~ return$~ ~ F_n^d ~ ~ \leftarrow ~ ~  \FILL(Z_{n-1}^{d-1},V
\setminus \{v_n\})  -
\Cone(v_n, F_{n-1}^{d-1}) $
}
\\ \\
To make the above generic construction explicit, it remains to specify how to choose
the pivot vertices, and the choice of the retuned $v$ in the base case
of $0$-dim filling. We will prove that regardless of this choice, the
output is an acyclic filling of $Z_n^{d-1}$. A good choice of the
pivot vertex will guarantee a lagre size filling.

Before presenting a formal proof we start with the analysis of the
procedure in the case $d \leq 1$ and $\F = \F_2$, which could also be taken a base
case for the inductive proof for $\F_2$ ahead. In this case we replace
$+,-$ over $\F$ with the  mod two addition $\xor$. Note also that for any
complex $A \subseteq K_n^r$ and $v \in V$, $A - Star(v,A) = A
\setminus \{v\}  = \{\sigma \in A| ~ v \notin \sigma \}$.
Note also that $\FILL()$ has formally a parameter indicating the
underlying set 
in respect to which the filling is created, and with respect to which
the deficit is defined. In what follows we drop this parameter from
the recursive call when ever it is clear from the context.

For $d=0$, the unique $(-1)$-dim nonzero cycle is $Z_n^{-1} =
\emptyset$. In this case for any vertex $v \in V,$ the chain $1
\cdot \{v\}$ namely, the singleton $v$,  is acyclic with
boundary $\emptyset$. 

For $d=1$, a non zero $0$-cycle $Z_n^{0}$ is a non-empty even-size subset of
$V$. In this case an acyclic filling of $Z_n^0$ is a forest $F
\subset K_n^1$ whose odd degree vertices is exactly the vertices in $Z_n^0$.  The
existence of a $0$-deficit filling in this case can be proven directly
from simple combinatorial consideration. In particular for
$Z_n^0=\{u,v\}$ this is any  path in $K_n^1$ whose end
points are $u,v$. 

Still, let us analyse the procedure for $d=1$, namely for an even
size set $Z_n^0 \subseteq [n]$: let
$v = v_n\in V(Z_n^0)$ be the chosen pivot vertex. Then
$\Link(v_n, Z_n^{d-1}) = \emptyset = Z_{n-2}^{-1}$ hence $F_{n-1}^{0}=
u \in V\setminus \{v\}$.  For any such $u,  ~ Z_{n-1}^0 = Z_n^0 \xor \{u,v\}$ is an even set. Either
$Z_{n-1} = \emptyset$ (in the case $Z_n^0 = \{v,u\}$) in which case
the acyclic filling $\Cone(v_n, u) = (v,u)$ is retuned. Otherwise, if
$Z_n \neq \{u,v\}$ or $u \notin
Z_n^0$ is chosen,  $Z_{n-1}$ is a non-empty even subset of $V
\setminus v$ . In this case a forest $F_{n-1}^0$ whose odd vertices is
returned as $\FILL_{n-1}^0(Z_{n-1}^0,V \setminus \{v\})$ and $F_{n-1}^0
\cup \{(v,u)\}$ is the final answer. Note that by induction (with the
right choice of $u$ above, namely $u \neq x$ in the case $Z_n^0
= \{v,x\}$) $F_{n-1}^0$ being $0$-deficit forest is of size $n-2$  resulting in
$F_n^0$ of size $n-1$, namely being $0$-deficit.

The analysis for $\F = \Q$ is similar and will be skipped.

We end this analysis of the case $d=1$ with the following claim that
will be used later.

\begin{claim}
\label{cl:many}
For any fixing of $v_n$ in the call for $\FILL(Z_n^0,V)$, there are at least $(n-2)!$ different
(labeled) $0$-deficit $1$-fillings of (any) $Z_n^0$.  In particular,
for $n\geq 4$ there at least two different fillings.
\end{claim}
\begin{proof}
  In the case $Z_n^0 = \{v,x\}$ there are $n-2$ choices of $u \notin
  Z_n^0$ that form a right choice of $u$ as described above. Each will
  correspond to a different final $F_n^1$ as for different $u,u'$
  $\Cone(v_n,F_{n-1}^{-1})$ contains only $(v_n,u)$ or $(v_n,u')$
  respectively. In the case $|Z_n^0| > 2,$ $u$ is unrestricted and can
  take any of the $n-1$ possible values. Hence the claim follows by
  induction and the observation that for $n=3$ there is $1$ such
  filling.

Again, the argument above is made formally for $\F_2$ but a similar
argument is done w.r.t $\Q$.  
\end{proof}

Before we prove Theorems \ref{thm:d2} and \ref{th:d>3} we first prove that for any field, the procedure returns an acyclic
filling.

\begin{lemma}
  \label{lem:082}
Let $Z_n^{d-1}$ be any non-zero cycle in $K_n^d$. Procedure
$\FILL (Z_n^d,V)$ returns an acyclic filling $F_n^d = \Fill(Z_n^{d-1})$ regardless of the
choice of $v_n$.  Further, 
 $\deficit (F_n^d) = \deficit (F_{n-1}^{d-1}) + \deficit (\FILL_{n-1}^{d-1}(Z_{n-1}^{d-1}, V \setminus \{v_n
\}))$, where $ F_{n-1}^{d-1}$ is any acyclic filling $F_{n-1}^{d-1} =
\Fill(Z_{n-1}^{d-2}, V \setminus \{v_n\})$, and $Z_{n-1}^{d-2},
Z_{n-1}^{d-1}$  are the corresponding objects
as defined  in the procedure.
\end{lemma}

\begin{proof}
The statement is obviously correct for $d=0$. Assume inductively that it is
correct for all $d'<d$ and for $d$ with $n' < n$.

First, let us verify that $Z_{n-1}^{d-1}$ is  a $(d-1)$-cycle
as otherwise the procedure is not even well defined. Indeed, since
$Z_n^{d-1}$ is a cycle, then $Z_{n-1}^{d-2}= \Link(v_n, Z_n^{d-1})$ is
a cycle as shown is Section \ref{sec:less_common}. Hence by induction it follows that
$\partial F_{n-1}^{d-1} = Z_{n-1}^{d-2}$. In addition, by Equations
(\ref{eqnConeLink}) and (\ref{eqnBoundaryLift}), $\partial (Star(v_n, Z_n^{d-1})) =
\Link(v_n, Z_{n}^{d-1}) = Z_{n-1}^{d-2}$. Plugging this into the expression for
$Z_{n-1}^{d-1}$ and taking its boundary it follows that

\[
\partial_{d-1} Z_{n-1}^{d-1} ~~=~~  \partial_{d-1} Z_n^{d-1} -  \, \Link(v_n,Z_n^{d-1}) +  
\, Z_{n-1}^{d-2}
~=~ 0  - Z_{n-1}^{d-2} + Z_{n-1}^{d-2} = 0
\]

Next, we show that $F_n^d$ is a filling of $Z_n^{d-1}$. 
Indeed, 
\[
\partial_d F_n^d ~=~  \partial_d \FILL(Z_{n-1}^{d-1},V
\setminus \{v_n\})  - \partial_d \Cone(v_n, F_{n-1}^{d-1})  = \]
\[
Z_{n-1}^{d-1} -   F_{n-1}^{d-1} + \Cone(v_n, \partial F_{n-1}^{d-1}) ~=~
Z_{n-1}^{d-1} - F_{n-1}^{d-1} + \Star(v_n,Z_{n-1}^{d-1}) ~=~ Z_n^{d-1} \]
 where the 2nd equality is by Equation (\ref{eqnBoundaryLift}), the next
 is by induction, and the last is by the definition of
 $Z_{n-1}^{d-1}$.

It remains to show that $F_n^{d}$ is acyclic.
Again, by induction this holds for $F_{n-1}^{d-1}$ and $F_{n-1}^d=
\FILL(Z_{n-1}^{d-1}, V \setminus \{v_n\})$. Hence this directly follows
from Claim \ref{cl:cone_extension}. 

Finally, $|\supp(F_n^d)|  = |\supp(\Cone(v_n,F_{n-1}^{d-1}))| +
|\supp(\FILL_{n-1}^{d-1}(Z_{n-1}^{d-1}, V \setminus \{v_n \})|$ since
these supports are disjoint. Further
$|\supp(\Cone(v_n,F_{n-1}^{d-1}))| = |\supp(F_{n-1}^{d-1}))|$,
hence, using that fact that ${n \choose k} = {n-1 \choose k-1} + {n-1
  \choose k}$,  it follows that  $\deficit (F_n^d) = \deficit(F_{n-1}^{d-1}) + \deficit(\FILL_{n-1}^{d-1}(Z_{n-1}^{d-1}, V \setminus \{v_n
\}))$ as claimed.
\end{proof}

\subsection{Proof of Conjecture~\ref{conj:main} for $d= 2$} 
\label{subsec:almost-HAM-2}

\subsubsection{Filling over $\F_2$}
We prove here the following restatement of Theorem \ref{thm:d2}
over $\F_2$.
\begin{theorem}
  \label{thm:d2_2}
Let $n\geq 4$. Every nonzero $1$-cycle $Z_n^1$ on $K_n^2$ has at least two  acyclic
fillings of deficit at most $1$ over $\F_2$. Further if the  the parity condition
holds it has a $0$-deficit acyclic filling, and for $n\geq 5$ it has
at least two such fillings. 
\end{theorem}
\begin{proof}
 The proof is by induction on $n$ . The case of $n=3$ is trivial.  For $n=4$, if
 the cycle is of length $3$, the parity condition holds (and there is
 a unique $0$-deficit filling). If the cycle is of length $4$ the
 parity condition does not hold and  there are two $1$-deficit
 fillings. For $n =5$ there are two
 cycles that meet the parity condition, each has at least two
 $0$-deficit filling. This can be easily checked by the reader.

We assume that the theorem is correct for any $Z_{n-1}^1, ~ n-1\geq 5$.
Recall that $Z_n^{1} \; \xor \; \Star(v_n,
              Z_n^{1}) = Z_n^1 \setminus \{v\}$, namely the subgraph
              obtained from $Z_n^1$ by deleting the vertex $v$ and all
              simplices that contain it.
Assume that $n \geq 6$ and that the parity condition holds for the
given $Z_n^1$.
Let $v=v_n \in V(Z_n^1)$ be arbitrary. Then,  the procedure $\FILL$
sets $Z_{n-1}^{1} ~= ~
              (Z_n^{1} \setminus \{v_n\})
              ~ \xor \;F_{n-1}^{1} $, where $F_{n-1}^1 =
              \Fill(\Link(v_n,Z_n^1),~V\setminus \{v_n\}) $ is
              a $0$-deficit tree in $K_{n-1}^1$, namely
              over $V \setminus \{v_n\}$ of size $n-2$. This exists by
              Claim \ref{cl:many}, as explained in the preface of this Section. 

To complete the construction, namely, to be able to use the induction
hypothesis on $Z_{n-1}^1$,  we only need that $Z_{n-1}^1 \neq \emptyset$ and
that the parity condition is met for it (with $n' = n-1$).

Note that $Z_n^{1} \setminus \{v_n\} = A$ is fixed and fully determined from $Z_n^1$
              once $v_n$ is chosen.  Now, for $F_{n-1}^1$ we have $(n-3)!
              \geq 3$ different legitimate fillings by Claim
              \ref{cl:many}. Hence for at least two of them $Z_{n-1}^1 = A
              \xor F_{n-1}^1$ is not the trivial cycle as
              needed. Choose one specific such $F_{n-1}^1$. 

Finally, $|Z_{n-1}^1| = |Z_n^1| \xor |\Star(v_n, Z_n^1)| \xor (n-2)
~(\bmod 2)$. Note that $|\Star(v_n, Z_n^1)| \equiv 0 (2)$ as $Z_n^1$
is a $1$-cycle. It follows that $|Z_{n-1}^1| \equiv |Z_n^1| - (n-2) \equiv {n-1
  \choose 2} -(n-2) \equiv {n-2 \choose 2} ~ (\bmod 2)$.
Where the 2nd equality is by the fact that the parity condition holds
for $Z_n^1$. Hence the parity condition holds for $Z_{n-1}^1$.

To show that there are at least two such fillings, we use the
induction on $n$. Namely,  by induction there are at least two
$0$-deficit fillings $\Fill(Z_{n-1}^1)$ for the fixed
$Z_{n-1}^1$. These two fillings result in two distinct fillings in the
return statement using the chosen fixed $F_{n-1}^1$. 

For the case that the parity condition does not hold,  the same
argument as in the last  two paragraphs implies that the parity condition does not
hold for $Z_{n-1}^1$ too. Hence again by induction we get at least two
$1$-deficit filling as the deficit of $F_{n-1}^1$ is $0$.
\end{proof}

\subsubsection{$d=2$ over $\Q$}

A analog of Theorem \ref{thm:d2_2} for $\F = \Q$  is similar except
that there is no parity obstacle. On the other hand, the induction base
cases for $n \leq 5$ are different. 

\begin{theorem} 
\label{theoremQtrees}
Let $n \geq 4$. For any nontrivial $1$-cycle $Z_n^1$ over $\Q$  there exists a $0$-deficit
$2$-filling $F_n^2$ except for the following two cases (the cycles
$C_i$'s below are directed, and uniformly weighted). 

$n=4$ and $Z_n^1 = C_4$ 

$n=5$ and $Z_n^1= C_3$.

Further, if
$n\geq 6$ every $1$-cycle has at least two such fillings.
In all the exceptional cases there exists 2-fillings of deficit 1.
\end{theorem}

\begin{proof}
  Assuming by induction that a $0$-deficit filling for $6 \leq n' < n$
  exists for every non-trivial $1$-cycle $Z_{n'}^1$, the proof for
  such filling for $Z_n^1$ is immediate and identical to the proof of
  Theorem \ref{thm:d2_2} (with addition over $\Q$ replacing $\xor$).

For $n \leq 6$ a case analysis is  presented in Appendix section \ref{sec:d2-Q}.
\end{proof}

%
%
\section{Proof of Theorem~\ref{th:d>3}}
\label{sec:d>3}
Fillings based on procedure $\FILL$ are not adequate to proof
Conjecture~\ref{conj:main}. The recursive call, even for $d=3$ uses 
filling for $d=2$ in the top level, which may not be $0$-deficit due to the
parity obstacle in the case of $\F_2$ (which is not an obstacle at all
for $d=3$), or due to the bad base cases for $\F=\Q$.   

An application of Theorem \ref{theoremQtrees} directly imply a filling
for $Z_n^{d-1}$ over $\Q$ of deficit $O(n^{d-3})$, see Section \ref{sec:over-q}.

A similar application of Theorem \ref{thm:d2_2} would 
 imply a filling for $Z_n^{d-1}$ of deficit $O(n^{d-2})$ over $\F_2$.  We aim
 however for the  same bound as for $\Q$. For this we will need to treat the
 case $d=3$ more carefully for $\F_2$. This will be done in the
 following Section
 \ref{sec:over-2}.

\subsection{Fillings over $\F_2$}\label{sec:over-2}
 We aim here to prove a slightly stronger results for
 $d=3$ and $\F_2$. It asserts that  a deficit of at most $1$ can
 always be achieved,  and a $0$-deficit can also be achieved for
 a large collection of cycles called {\em friendly cycles} below.

\ignore{
\begin{theorem}
\label{th:d>3_f2}
Let $Z_n^{d-1}$ be a nonempty  $(d-1)$-cycle over $\F_2$ on $K_n^d$. Then 
there  exists an acyclic filling $F_n^d$  of $Z_n^{d-1}$ of deficit
$O(n^{d-3})$. 
In particular, for $d=3$, the deficit is constant.
\end{theorem}

The proof is by induction on $n$. Here the base case
for $d=3$ is the hard part, while for $d > 3$ the proof is trivial
and based on it. We start with the following theorem for $d=3$, for
which we also need the following definition.
}

Recall that for a chain $C \subseteq K_n^d$ and a vertex $u \in [n],~
deg(u,C) = |\Star(u,C)|$ namely, it is the number of $d$-simplices in
$C$ that contain $u$.

\begin{definition}[friendly cycle]
  A cycle $Z_n^2$ is called friendly if there exist two vertices
  $v',v'' \in V(Z_n^2)$ such that $deg(v',Z^2_n) \not\equiv deg(v'',
  Z_n^2) ~ (\bmod ~2)$.
\end{definition}

\begin{theorem}
  \label{thm:d=3_f2}
Let $Z_n^{2}$ be a friendly  $2$-cycle over $\F_2$ on $K_n^2$. Then 
there  exists an acyclic filling $F_n^3$  of $Z_n^{2}$ of
$0$-deficit. Moreover, if $n \geq 7$ there are at least $2$ such fillings.
\end{theorem}

\noindent
{\bf A matter of notations:} The recursion call for $\FILL(Z_n^2,V)$
results in a double recursion:  one for the lower dimensional
$\FILL(Z_{n-1}^1, ~ V')$ and the other is for
$\FILL(Z_{n-1}^2, ~ V')$, where $V' = V\setminus \{v_n\}$. For the latter, all arguments
will be determined by the induction process. For the former, in order make the notations less
cumbersome  we remove $V'$ from  $\FILL(Z_{k}^1,V')$ and just write
$\FILL(Z_k^1)$. The subscript $k$ defines the current $|V'|$ (for a
filling of a $1$-dim cycle) and its
actual value is $V' = V \setminus \{v_n, v_{n-1}, .... v_{k+1} \}$ for
the implicitly defined pivot vertices $\{v_n,..., v_{k+1}\}$. 

Before proving the theorem we first start with an explicit expression for the degree of a vertex
in $Z_{n-1}^2$, where $Z_{n-1}^2$ is the cycle generated by the call
of $\FILL(Z_n^2, ~ V\setminus \{v_n\} )$ at the top level recursion. This will be used
later to see how the degree of a vertex w.r.t $Z_{n-i}^1$ evolves in the recursion.

\begin{claim}
  \label{cl:friendly_internal}
Let $Z_{n-1}^2$ be as defined by $\FILL(Z_n^2,V)$ using $v_n$ as the pivot
vertex at the top recursion call. Let $u \in [n] \setminus \{v_n\}$. 
Then $deg(u,Z_{n-1}^2) = A(u) \xor B(u)$ where $A(u)=deg(u,Z_n^{2} \; \xor \; \Star(v_n,
              Z_n^{2}))$ depends only on
$Z_n^2$ and $v_n$ but not the implementation  of $\FILL$ in the lower
recursion levels. 
$B(u)= deg(u,\FILL(Z_{n-1}^{1}))$ depends on whether $u=v_{n-1}$ in the recursive call for
$\FILL(Z_{n-1}^1)$ or not.  

If $u=v_{n-1}$ we have
$B(u) = n-3$. 

Otherwise $$B(u) \equiv deg(u,\FILL(Z_{n-2}^1)) \xor
deg(u,\Link(v_{n-1},Z_{n-1}^1)) ~ ~ ~ ~ (\bmod ~2)$$
\end{claim}
\begin{proof}
  Recall that by the definition of $\FILL(Z^2_n)$ with respect to $v_n$
  being the pivot, $$Z_{n-1}^{2} =
              Z_n^{2} \; \xor \; \Star(v_n,
              Z_n^{2})\; \xor \;F_{n-1}^{2}$$ where
              $F_{n-1}^2 = \FILL(Z_{n-1}^1)$ and $Z_{n-1}^1 = \Link(v_n,Z_n^2))$.

Hence,               $$deg(u,Z_{n-1}^{2}) \equiv
               deg \left [ ((u,Z_n^{2}) \xor \Star(v_n,Z_n^{2}))
              \right ]\; \xor \;deg(u,F_{n-1}^{2}) \equiv A(u) \xor B(u) ~
              ~ ~ (\bmod ~2)$$

Now obviously $A(u)$ depends only on $Z_n, v_n$ but not on the
implementation of $F_{n-1}^2$.

$B(u) \equiv deg(u,F_{n-1}^{2}) \equiv deg(u, \FILL(Z_{n-1}^1)) ~ ~
~ 
~ (\bmod ~2)$.

Recall that using $\FILL$ recursively $\FILL(Z_{n-1}^1) = \FILL(Z_{n-2}^1)
\xor \Cone(v_{n-1}, \FILL(\Link(v_{n-1}, Z^1_{n-1})))$. Recall also that
$\Link(v_{n-1}, Z^1_{n-1})= Z_{n-2}^0$ is $0$-dim cycle namely, an even set of
vertices and hence $\FILL(\Link(v_{n-1}, Z^1_{n-1}))$ can be
implemented to result in a $0$-deficit
tree $T_{n-2}$ on $[n-2]$, whose set of odd vertices is $Z_{n-2}^0$.

  If $u=v_{n-1}$ in the call for
$\FILL(Z_{n-1}^1),$ $v_{n-1} \notin V(\FILL(Z_{n-2}^1))$ while it
forms a $2$-simplex with every edge of $T_{n-2}$, namely with $n-3$ edges.
 Hence the claim follows in
this case.

If $u \neq v_{n-1}$ then by definition of $B(u) \equiv deg(u,\FILL(Z_{n-2}^1)) \xor
deg(u,\Cone(v_{n-1},T_{n-2})) $, where $T_{n-2}$ is a tree as above.
 But $deg(u,\Cone(v_{n-1},T_{n-2}) = deg(u,T_{n-2})  = deg(u,\Link(v_{n-1},Z_{n-1}^1)) $ and the
claim follows.
\end{proof}

The core of the argument in the proof of the theorem is to analyze how the
parity condition of $Z_{n-1}^1$ depends on $Z_{n}^1$ and the vertex
$v_{n}$ that is chosen to be the pivot in the top level call of $\FILL$. It is shown
next, that regardless of $Z_n^2$ and $v_n$ 
that  determine $Z_{n-1}^1$, the freedom in the construction of
$F_{n-1}^1$ in the top call of $\FILL$ is enough to guarantee that
$Z_{n-1}^2$ will be friendly.

\begin{lemma}
  \label{lemma:friendly}
Let $n \geq 7$, $Z_n^2$ a non empty $2$-cycle and $v_n \in V(Z_n^2)$. Then
there is $F_{n-1}^2 = \FILL(\Link(v_n, Z_n^2))$ as guaranteed by
Theorem  \ref{thm:d2_2} such that 
$Z_{n-1}^2$ that is produced by the call $\FILL(Z_n^2)$  using
$F_{n-1}^2$ in the top recursion level is a  friendly
cycle.  

 Further, if $\Link(v_n,Z_n^2)$ is friendly, then there are at
least two distinct such
$0$-deficit fillings $\Fill(\Link(v_n,Z_n^2))$. 
If
$\Link(v_n,Z_n^2)$ is not friendly then there are two distinct $1$-deficit fillings as above.
\end{lemma}

\begin{proof}
Let $Z_{n-1}^1 = \Link(v_n,Z_n^2)$  be the $1$-cycle that is defined in the
call of procedure $\FILL(v_n,Z_n^2)$. Let $F_{n-1}^{2} =
\FILL(Z_{n-1}^{1}) $ and  $Z_{n-1}^{2} =  Z_n^{2} \xor  \Star(v_n,
Z_n^{2}) \xor \;F_{n-1}^{2}$.  To prove the claim it is enough to show
that $F_{n-1}^2$
can be constructed so that (a) there are two vertices $x,y \in V(Z_{n-1}^2)$ for which
$deg(x,Z_{n-1}^2) \not\equiv deg(y,Z_{n-1}^2) ~ ~ (\bmod ~2)$, (b) that
$F_{n-1}^2$ is $0$-deficit or $1$-deficit depending on whether $Z_n^2$
is friendly or not, correspondingly, and (c) - that two such distinct
$F_{n-1}^2$ can be constructed for each case.

 Consider the following cases:

\noindent
{\bf Case 1:} there are $u,u' \in V(Z_{n-1}^1)$ such that $A(u) \not\equiv
A(u')~ (\bmod ~2)$ and $(u,u') \in Z_{n-1}^1$. Here $A(v)$ is as defined in Claim \ref{cl:friendly_internal}.

In that case we choose $u = v_{n-1}$ in the definition of
$F_{n-1}^2=\FILL(Z_{n-1}^1)$, and $u' = v_{n-2}$; namely the pivot vertex in the
call of $\FILL(Z_{n-2}^1)$ which is made in the next recursion level
call in the construction of  $F_{n-1}^2=\FILL(Z_{n-1}^1)$.
We will need to show that $u' \in V(Z_{n-2}^1)$ for
this to be possible. Assume for now  that $u' \in V(Z_{n-2}^1)$.

Claim \ref{cl:friendly_internal} implies that 
\begin{equation}
  \label{eq:3.5-1}
  deg(u,Z_{n-1}^2) ~\equiv A(u) \xor B(u) \equiv~ ~  A(u) \xor  n-3 ~
(\bmod ~2)
\end{equation}
Also, by the same Claim, 

 \begin{equation}
   \label{eq:3.5-2}
   deg(u',Z_{n-1}^2) ~\equiv A(u') \xor B(u') \equiv~ ~
   deg(u',\FILL(Z_{n-2}^1)) \;\xor \; 
 deg(u',\Link(v_{n-1},Z_{n-1}^1)) ~ ~ (\bmod ~2)
 \end{equation}
 Since $v_{n-2}=u'$,
 reapplying Claim \ref{cl:friendly_internal} w.r.t $u'$ and
 $Z_{n-2}^1$,  we get $deg(u',\FILL(Z_{n-2}^1)) \equiv n-4
 ~ (\bmod ~2)$.

 Since $(u,u') \in Z_{n-1}^1$ we have that $u' \in
 \Link(u,Z_{n-1}^1)$ namely $deg(u',\Link(v_{n-1},Z_{n-1}^1)) \equiv 1
 ~ (\bmod ~2)$.

 Plugging the above into Equation (\ref{eq:3.5-2}) and using that
 $A(u) \not\equiv A(u')$,   we conclude that 
$deg(u,Z_{n-1}^2) \not\equiv deg(u',Z_{n-1}^2)$, namely that 
$Z_{n-1}^2$ is friendly.

Further, Theorem \ref{thm:d2_2} asserts that
$F_{n-1}^2$ can be made $0$-deficit if $Z_{n-1}^1$ meets the parity
conditions, and of deficit $1$ otherwise.

To conclude this case what is left to be shown is that we can
construct $Z_{n-2}^1$ such that $u' \in V(Z_{n-2}^1)$. This is done
using the relatively large freedom we have in constructing
 $Z_{n-2}^1$. The argument is  formally 
presented in Claim \ref{cl:case1},  Appendix \ref{app:d=3}.  Finally,
this construction will result in one $F_{n-1}^2$ as needed. To
construct a different one with the same properties it is enough to
exchange the roles of $u,u'$ in the construction above. It is left for
the reader to realize that this will  result in a different
$F_{n-1}^2$ (as in particular $u,u'$ will have different degrees with
respect to $Z_{n-1}^2$ in the two constructions).

\noindent
{\bf case 2:} Assuming that Case 1 does not happen then in every component of
$Z_{n-1}^1$ every two vertices $x,y$ have  $A(x)
\equiv A(y) ( \bmod ~2)$.

If
 there are
$u,u'$ with $A(u) \equiv A(u') ~ (\bmod ~ 2)$ but $(u,u') \notin
Z_{n-1}^1$, then choosing $u=v_{n-1}$ we get $B(u)=n-3$. We show in Claim
\ref{cl:case_2.1} in  Appendix \ref{app:d=3}  that $Z_{n-2}^1$ can be constructed
so that $u'\in V(Z_{n-2}^1)$. Hence choosing $u'=v_{n-2}$ implies that
$B(u') \equiv (n-4) + deg(u',\Link(u,Z_{n-1}^1)) \equiv (n-4)
~ (\bmod 2)$ on account that $(u,u') \notin Z_{n-1}^1$. We conclude
that $deg(u,Z_{n-1}^2) \not\equiv deg(u',Z_{n-1}^2)$ and hence
$Z_{n-1}^2$ is friendly. 

Further $F_{n-1}^2$ is of $0/1$-deficit as needed as in the previous
case. In addition, exchanging the roles of $u,u'$ will result in a
different $F_{n-1}^2 = \Fill(Z_{n-1}^1)$ with the same desired properties, by a
similar argument as in the previous case.

\noindent
{\bf case 3:}
We are left with the case that neither case 1, nor case 2 occur. In
this case either 
$Z_{n-1}^1$ is the complete graph on $[n-1]$ and is monochromatic
w.r.t. $A(*)$, or $Z_{n-1}^1$ is a
union of  two  cliques, each being monochromatic w.r.t. $A(*)$ and
with different values of $A(*)$ in these two cliques. 
This very special case is analysed in Claim \ref{cl:case_2_d=3} in
Appendix \ref{app:d=3}. It  asserts that in this case too $Z_{n-2}^1$
can be made friendly. Further two corresponding $F_{n-1}^2$  of
$0/1$-deficit are constructed as needed.
\end{proof}

\begin{proof}[of Theorem \ref{thm:d=3_f2}]

The proof is by induction on $n$. 
The base case is for $n\leq 7$ which we have checked by a computer
program see  Appendix \ref{cl:thm:d=3_f2_n<6}. The Theorem is in fact true for $n=6$, but
we have stated it for $n\geq 7$ so to use one computer program for
every cycle (friendly or not) - see Theorem \ref{thm:d=3_f2-x}.

Let $n \geq 8$ and let $Z_n^2$ be a friendly cycle. Let $v \in V(Z_n^2)$ for which
$Z_{n-1}^1 = \Link(v,Z_n^2)$ meets the parity conditions. Such $v$ exists by the
assumption of $Z_n^2$ being friendly. Set $v=v_n$ and use the
procedure $\FILL$ with $v_n$. This will produce a filling 
$F_n^3 =   \FILL(Z_{n-1}^2)  \xor
\Cone(v_n, F_{n-1}^{2}) ~ )$, where  $F_{n-1}^2 = \Fill(Z_{n-1}^1)$ 
 is as guaranteed by Lemma \ref{lemma:friendly} to result in a
 friendly $Z_{n-1}^2$. Hence by induction $\FILL(Z_{n-1}^2)$ can
 produce two distinct $0$-deficit fillings resulting in two distinct
 fillings for $Z_n^2$.   

Since $F_{n-1}^{2}$ is guaranteed to be $0$-deficit by Theorem
\ref{thm:d2_2}, and $Z_{n-1}^2$ is friendly, this implies that
$F_n^3$ is $0$-deficit by induction and Lemma \ref{lem:082}.  
\end{proof}

Theorem \ref{thm:d=3_f2} immediately implies the following more
general theorem.

\begin{theorem}
    \label{thm:d=3_f2-x}
Let $n \geq 7$ and $Z_n^{2}$ be a nonempty  $2$-cycle over $\F_2$ on $K_n^2$. Then 
there  exist at least two  acyclic filling $F_n^3$  of $Z_n^{2}$ of
deficit that is at most $1$. 
\end{theorem}

\begin{proof}
The proof is again by induction on $n$. For $n \leq 7$ it follows by
checking finitely many possible cycles which was done by a computer
program, see Appendix  \ref{sec:program}. 
 If 
 $Z_n^2$ is friendly  the assertion follows by from Theorem
\ref{thm:d=3_f2}.  

Assume that $Z_n^2$ is not friendly, and $n \geq 8$.
Assume that for some $v \in
V(Z_n^2)$, $\Link(v, Z_n^2)$ meets the parity condition. Then by Lemma
\ref{lemma:friendly} with respect to $v = v_n$, there is a $0$-deficit
$F_{n-1}^2
= \FILL(\Link(v_n, Z_n^2))$, such the resulting $Z_{n-1}^2$ in the top
recursion level of $\FILL(Z_n^2)$ is friendly. Then by Theorem
\ref{thm:d=3_f2} there are two $0$-deficit fillings $F,F'$ each being
a $0$-deficit filling of $Z_{n-1}^2$. Using each in the top call for
$\FILL(Z_n^2)$ together with $F_{n-1}^2$ we get two corresponding
$0$-deficit fillings for $Z_n^2$.

If $Z_n^3$ is not friendly, we pick an arbitrary $v_n \in V(Z_n^2)$ as a pivot vertex
used in the top recursion   level in $\FILL$. Then Lemma
\ref{lemma:friendly} asserts that $F_{n-1}^2$ will be a $1$-deficit
filling and that $Z_{n-1}^2$ will be friendly. Hence Theorem
\ref{thm:d=3_f2} asserts at least two $0$-deficit filling of $Z_{n-1}^2$
resulting in at least two $1$-deficit filling of $Z_n^2$.
\end{proof}

\subsection{Fillings in dimension larger than $3$}\label{sec:over-q} 

To prove Theorem \ref{th:d>3} our intension is to use induction on the
pair $(d,n)$. The base case for $d
\leq 2$ and any $n$ is proved in Theorem
\ref{theoremQtrees} for $\Q$ and in Theorem \ref{thm:d=3_f2-x} for
$\F_2$ and $d \leq 3$.
We will need a base case for every $d\geq 3$ and some
small $n=n_d$. This is shown in the next claims.

\begin{claim}
  \label{cl:nq}
Let $n= d+2$ and $Z_n^{d-1}$ be a non-empty cycle over $\Q$. Then there are two
distinct fillings for $Z_n^{d-1}$, each of deficit at most $d$.
\end{claim}
\begin{proof}
 Every
  $(d-1)$ cycle  $Z=Z_n^{d-1}$ can be written as $Z= \sum_{\sigma \in K_n^d} \alpha_\sigma
  \cdot \partial_d \sigma$, where $\alpha_\sigma \in \Q$ and 
the support of the this sum, $F =\{\sigma ~ | ~ ~ \alpha_\sigma \neq
0 \}$, is not empty. 

Assume first $F \neq K_n^d$, namely that there is
$\tau \in K_n^d \setminus F$. Note that the expression for $Z$ defines
a filling of $Z$ supported on $F$. Further $F$ is acyclic as $|F| \leq {n \choose d+1} -1 =
d+2-1 = d+1$ and the smallest $d$-cycle is of size $d+2$. 

Now to get another acyclic filling, replace for some $\sigma \in F$
the term $\partial \sigma$  with
$-\sum_{\sigma' \in K_n^d, ~\sigma'  \neq \sigma} \sigma' $ in the
expresion for $Z$.  Since
$\partial \sigma  = -\partial (\sum_{\sigma' \in K_n^d, \sigma'  \neq
  \sigma} \sigma' )$ we get again a filling $F'$ of $Z$.   Note that
$\sigma \in F \setminus F'$, where $F'$ is a new support after the
above substitution. In particular $F' \neq F$. Hence the new sum is indeed a
different filling. Further $F'$ is acyclic by the same reasoning as
above, on account of $\sigma \notin F'$ which implies that $|F| \leq d+1$.

If $F = K_n^d$ then up to scaling we may assume that for 
$\sigma = (2, 3, \ldots d+2),~ \alpha_\sigma = 1$. In that case either for every $\tau
\in K_n^d$, $\alpha_\tau$ is identical to the coefficient of $\tau$ in
$\partial _{d+1}(1,\ldots ,d+2)$. In this case $Z
= \partial_d \partial_{d+1} (1, \ldots ,d+2)$ is the trivial cycle.  We
conclude that for some $\tau$, $\alpha_\tau$ is not identical to the
coefficient as defined above. Now one can cancel $\sigma$ from the sum
representing $Z$ 
by adding to the sum expressing $Z$ the expression  $-
\partial_{d+1}(1, \ldots, d+2)$ which is $0$. But $-\partial_{d+1}(1,
\ldots, d+2)$ includes $\sigma$ with coefficient $-1$ and
will cancel $\sigma$ from the sum. 
 Hence, this new sum (of support at moset $d+1$) is an acyclic filling
 of $Z$.

Alternatively getting another acyclic filling is by adding to $Z$ the sum
$-\alpha_{\tau} \partial_{d+1} (1, \ldots ,d+2)$ which will cancel $\tau$ but will not
cancel $\sigma$.

Finally, as the
rank is $d+1$,  the deficit of the fillings is obviously at most $d$.
\end{proof}

A similar claim for $\F_2$ is as follows.
\begin{claim}
  \label{cl:n0}
Let $n= d+2$ and $Z_n^{d-1}$ be a non-empty cycle. Then there are two
distinct fillings for $Z_n^{d-1}$,  each of deficit at most $d$.
\end{claim}
\begin{proof}
  The proof is almost identical to that over $\Q$, except that the
  case of $F = K_n^d$ in sum expressing the cycle $Z$. In this later
  case, since all non-zero coefficients are $1$, we have that
  $Z=Z_n^{d-1} = \sum_{\sigma \in K_n^d } \partial \sigma$. But this
  is just $0$ (on account of $\partial\partial (1, \ldots d+2) =
  0$). Namely, this case does not need any attention as $Z$ is the
  trivial cycle.
\end{proof}

We now prove the following stronger theorem that implies Theorem \ref{th:d>3}.
\begin{theorem}
  \label{thm:d>3_f2_strong}
There exists a function $c: \N \mapsto \N,  ~ ~ d \longrightarrow c_d$
such that for every  nonempty  $(d-1)$-cycle $Z_n^{d-1}$ over $\F_2$
or over $\Q$, on $K_n^d$, 
there  exist at least two acyclic filling of $Z_n^{d-1}$ each
of deficit at most 
$ c_d \cdot n^{d-3}$. 
\end{theorem}

\begin{proof}

The proof is by induction on the pair $(d,n)$. For $\F_2$,  $d\leq 3$ and every
  $n$ it follows from Theorem \ref{thm:d=3_f2-x} and Theorem
  \ref{thm:d2_2}.  For every $d$ and small enough $n$ it follows from
  Claim \ref{cl:n0}.  Similarly,  for $\Q$ and $d \leq 2$ it follows
  from Theorem \ref{theoremQtrees}. Further, for $d \geq 3$ and small
  enough $n$ it follows from Claim \ref{cl:nq}.

  The induction now is identical for both $F_2$ and $\Q$:

Let $Z_n^{d-1}$ be a non-empty $(d-1)$-cycle for $d \geq 4$ for $\F_2$
or $d \geq 3$ for $\Q$. Let $v
\in V(Z_n^{d-1})$ be arbitrary. Then applying $\FILL(Z_n^{d-1})$ with $v_n =v$
in the top level results in $Z_{n-1}^{d-2}$, the corresponding filling
$F_{n-1}^{d-1}= \FILL(Z_{n-1}^{d-2})$ by recursion, and
$Z_{n-1}^{d-1}$. Further, by the induction hypothesis we may assume
that $F_{n-1}^{d-1}$ is of deficit at most $c_{d-1} \cdot (n-1)^{d-4}$
 (or $0$ deficit if $d-1=2$ for $\Q$).
Fix one such filling that results 
in a non-empty $Z_{n-1}^{d-1}$ (there exists one on account of the existence of at
least two distinct fillings $F_{n-1}^{d-1}$ as above). We get by
induction at least two fillings for $Z_{n-1}^{d-1}$ each of size at
most $c_d \cdot (n-1)^{d}$.

Then the filling that is defined by $F_{n-1}^{d-2}$ and each of the
two fillings $F_{n-1}^{d-1}$ in the top level call of
$\FILL(Z_n^{d-1})$ results in a filling with deficit $c_d \cdot
(n-1)^{d-3} + c_{d-1} \cdot (n-1)^{d-4}$. Solving the recursion
obviously results in a $c_d \cdot n^{d-3}$ deficit filling.
\end{proof}

We end this section with the following conjecture that is weaker than
Conjecture \ref{conj:main}. It states that the procedure $\FILL$ can
  always be made to produce a filling with deficit that is independent
  on $n$ but may depend on $d$.

\begin{conjecture}
\label{conj:f2}
There exists a function $\alpha: \N \mapsto \N,  ~ d \mapsto \alpha_d$
such that for every non-trivial $(d-1)$-cycle $Z^{d-1}$ on $K_n^d$
(w.r.t. $\F_2$ or $\Q$), $\FILL(Z_n^{d-1})$ can be made to produce a filling of
deficit at most $\alpha_d$.
\end{conjecture}

\ignore{
  A way to prove Conjecture \ref{conj:f2} is to generalize the
definition of 'friendly' cycle to higher dimensions, and generalize
Lemma \ref{lemma:friendly} to show that $F_{n-1}^{d-1}$ can always be
produced so to result in a friendly $Z_{n-1}^{d-1}$ in the top level
call of $\FILL$.  One such generalization is an inductive one:  a
$d$-cycle $Z_n^{d-1}$ is a $d$-friendly cycle if contains a vertex $v$
such that $\Link(v,Z_n^{d-1})$ is $(d-1)$-friendly.  However, for this
``implicit'' definition we do not have a suitable generalization of the
corresponding lemma.  
}

\section{\boldmath{On the maximum size of a simple $d$-cycle on
    $[n]$}}\label{sec:large-cycles}

Here we use the results in Sections  \ref{sectionFillings} and
\ref{sec:d>3} to show the
existence of large simple $d$-cycles.  As explained in the
introduction, for the very simple case of $d=1$, Hamiltonian cycles,
namely simple cycle of the maximum possible size of $r(n,1)+1 = n$
exist for very $n\geq 3$.  For $d \geq 2$ this was  open.

Let  $\sigma$ be a $d$-simplex.  Recall that for an acyclic
$d$-filling $F^{(d)}$ of the $(d-1)$-cycle
$\partial\sigma$,   the $d$-chain $F^{(d)} - \sigma$ is 
a simple $d$-cycle. Conversely, for a simple $d$-cycle $Z$ and $\sigma
\in Z$, $Z - \sigma$ is an an acyclic $d$-filling of $\partial\sigma$.    
Thus, Theorem \ref{th:d>3}, immediately imply
the existence of large simple $d$-cycles in $K_n^d$ over $\F_2$ and
over $\Q$. This is not, however, likely to be tight.

The existence of the extreme case, that is, Hamiltonian cycles, or tighter
results are of particular interest. We next sum up the consequences of
Theorem \ref{th:d>3}  in Theorem \ref{th:hamiltonian} below.

\begin{theorem}
\label{th:hamiltonian} $\mbox{}$

For $d=2$, over $\F_2$, Hamiltonian $2$-cycles on $[n]$ exist  if and only if 
$n \equiv 0\, { or }\, 3\, (mod\, 4)$. Over $\Q$, they exist for all $n\geq 4$ with 
exception of $n=5$. In all cases there exist simple cycles of deficit $\leq 1$.

For $d \geq 3$, over $\F_2$ as well as over $\Q$, 
there exist simple $d$-cycles on $[n]$ with deficit  $O(n^{d-3})$.    $\qed$
\end{theorem}

We next consider $3$-dimensional cycles over $\F_2$. The tighter
Theorem \ref{thm:d=3_f2-x} immediately implies that there are simple
$3$-cycles of size $r(n,3) = {n-1 \choose 2}$, namely of size
$1$-short of being Hamiltonian. This by the discussion above, and the
fact that for a $3$-dim simplex $\sigma$, there is a $1$-deficit
filling of $\partial_3 \sigma$.

Note that for every $v \in Z_n^2$, $|\Link(v, \partial_3 \sigma)| =
3$, hence $\partial_3 \sigma$ is not friendly. Therefore  Theorem
\ref{thm:d=3_f2} is not applicable to yield  a tighter  $0$-deficit
filling of $\partial_3 \sigma$ and, in turn, a  Hamiltonian $3$-cycle.
  However, the only need of being friendly in the proof of Theorem
  \ref{thm:d=3_f2},  is to be able to choose $v=v_n$ for the top
  level call of $\FILL(Z_n^2)$, so that the parity condition holds for the $1$-dim
  cycle $Z_{n-1}^1 =  \Link(v_n, Z_n^2)$. In our case for $Z_n^2 =
  \partial_3 \sigma$, and as remarked above $\Link(v,Z_{n-1}^2) = 3$
  for every $v \in V(Z_{n}^2)$. Hence whenever
  ${{n-2} \choose 2} \equiv 1  ~({\rm mod}\;2)$ it has the parity
  condition, and $v$ could be taken so that $\FILL(Z_{n-1}^1)$ is
  $0$-deficit, resulting in a $0$-deficit filling of $Z_n^2$. This
  implies, in turn,  a Hamiltonian $3$-cycle. We sum this in the
  following corollary.
  
  \begin{corollary}
    \label{cor:ham-d=3}
    For every $n \equiv 0,1 ~ (\bmod 2), $ and $ n\geq 7$,  there is a
    $3$-Hamiltonian cycle in $K_n^3$ with respect to $\F_2$.
  \end{corollary}
  \begin{proof}
    For such $n$ the parity condition for the $2$ cycle $\partial_3
    \sigma$ with respect to $n-1$ holds, and hence there is a
    $0$-deficit filling of it resulting in a Hamiltonian cycle as explained above.
  \end{proof}

In what follows we focus no $\F_2$ and discuss some non-trivial upper bounds for the
largest simple cycles when $n$ is relatively small with respect to
$d$.

 By a standard duality argument (see e.g.,~\cite{BjornerTancer}), there is 
a size- and deficit-preserving 1-1 correspondence between the $(n-d-2)$-hypercuts 
(= simple $(n-d-2)$-cocycles), and the simple $d$-cycles in $K_n^{n-1}$. 
In~\cite{LNPR}, the authors discuss lower bounds on the deficit of $k$-hypercuts
in $K_n^k$. In particular, it holds that:
$\mbox{}$ {\em

The deficit of the largest $2$-hypercut in $K_n^2$ is $n^2/4 - O(n)$.

For any odd $k$,  the deficit of largest $k$-hypercut  
is at least ${{n-1} \choose k} \left( {{n} \over {(k+1)^2}} - 1 \right)$. 
(This holds for general $k$-cocycles as well.)
}

Combining these results with the above duality, and setting $k=n-d-2$, one
arrives at the following results about the deficits of $d$-cycles:
\begin{claim}
\label{cl:deficit}$\mbox{}$

The deficit of the largest simple $d$-cycle in $K_{d+4}^d$
is ${1\over 4}d^2 - O(d)$.

The deficit of any $d$-cycle in $K_{d+k+1}^d$, $k$ odd, is at least 
${{k+d+1} \choose k} \left( {{k+d+2} \over {(k+1)^2}} - 1 \right)$.
\end{claim} 
\begin{corollary} \label{cl:cor13}$\mbox{}$

For a large $d$ and an odd $k \approx \sqrt{d}-1$, the deficit of any
$d$-cycle in $K_{d+k+1}^d$ is at least $\; (d/e)^{0.5\sqrt{d} - O(1)}$. 
\end{corollary}
%
%
%
%
%

\section{Concluding remarks}
We have shown that for every $d$ and large enough $n$ there is a large
acyclic $d$-filling of any $(d-1)$-cycle.  For the case $d\leq 2$ this
is completely closed (over $\F_2$ and over $\Q$).  In particular, this
shows the existence of very large simple $d$-cycles.  The extremal case
of Hamiltonian cycle is solved completely for $d \leq 2$. For
$d=3$ over $\F_2$, we have shown  the existence of Hamiltonian cycles
for an infinite sequence of $n$'s. However, the existence of
Hamiltonian cycles for higher dimensions  is open at
large. Currently we do not even see a method of approaching the
problem. This poses one major open problem.

Other related open problems are proving either Conjecture
\ref{conj:main} or the weaker Conjecture \ref{conj:f2}. 

Another interesting point that follows from the discussion in this paper concerns the
existence of non-collapsible trees.

A $(d-1)$-simplex $\tau$ of a pure $d$-dimensional simplicial complex
$X$ is called {\em exposed} if its degree is $1$, that is, it belongs
to exactly one $d$-simplex $\sigma$ of $X$. An elementary $d$-collapse
on an exposed $\tau$ as above, consists of the removal of $\sigma$ and $\tau$ from $X$. The
complex $X$ is {\em collapsible to its $(d-1)$-skeleton} if every
$d$-simplex of $X$ can be removed by a sequence of elementary
collapses of $(d-1)$-facets. It is easy to see that if $X$ is
collapsible to its $(d-1)$-skeleton, then $X^{(d)}$ is
acyclic over any field.  Is the inverse true?  For $d=1$ this is true;
the fact that every acyclic graph is collapsible is  identical to
the fact that every non-empty acyclic graph contains a vertex of
degree $1$ (a leaf).

The existence of non-collapsible trees (over $\F_2$ and over $\Q$) was
known, cf. \cite{}. A consequence of our results is a
construction of non-collapsable $d$-trees for $d=2,3$. In fact the trees
that we construct do not have any exposed $d-1$ simplex. The way to
construct such trees, is to construct a Hamiltonian cycle $Z$, namely in
which no exposed $(d-1)$-simplex exists. Further, to observe that for some
$d$-simplex $\sigma$ in it, any $\tau \in \partial \sigma$ appears
with multiplicity at least $4$. Hence, removing $\sigma$ from $Z$ will
result a tree in which there is no exposed simplex.

{\bf Acknowledgement: } We wish to thank Ofer Magen and Yuval Salant for
writing and running the computer program that checked the $\F_2$-cases
of $n=6,7$ and $d=2$.  The source of this program can be found in
Appendix section\ref{sec:program}.


\clearpage
\appendix{\Large{\bf Appendix}}

\section{Case analysis for $d=2$ over $\Q$ and $n \leq 6$, for the
  base cases of Theorem \ref{theoremQtrees}}\label{sec:d2-Q}

\ignore{
We start with a remarks:
$Z_k^1$ never has a vertex of degree $1$, namely adjacent to
  only one edge of non-zero weight, as the edge weight cannot be
  cancelled. 

Further, let $Z_r^0$ be a $0$-cycle  over $V_r$ (this is a weighting
of $V_r$ that sum up to $0$). Let $T$ be any spanning tree $T$ over $V_r$.
Recall that there always exist a unique filling $F_r^1$ of $Z_r^0$ that is
supported on $T$. 
\begin{claim}
\label{cl:0-edge}
An edge $e\in T$ will have weight $0$ in $F_r^1$ if and only if the total weight of the {\em vertices} in (any) connected component of $T\setminus {e}$ sums up to $0$.
Consequently, if no such edge exists in $T$, then the corresponding
$F_r^1$ has $0$-deficit.
\end{claim}
\begin{proof}
  Working from the leaves of $T$ ``inwards'' the weights of a filling
  $F_r^1$ are determined. Hence if for some subtree the weights sum up to
  $0$, all edges in the corresponding cut will have weights $0$. The
  inverse is similar.
\end{proof}
}


We identify here a $1$-cycle $C$ with a weighted directed graph
$C=(V,E)$ in which for every $v \in V(C)$, $\sum_{(x,v)\in C} w(x,v) -
\sum_{(v,y)\in C} w(v,y) = 0$. 
\begin{description}
\item[$n=3$]

This case is essentially empty: The unique $1$-cycle is
$\partial(\sigma)$ for $\sigma$ being the unique $2$-dim
simplex. Hence $\sigma$ is the required $0$-deficit 
filling.

\item[$n=4$]  
In general all possible cycles (up to isomorphism and scaling of weights) are of the form $G_2
= \partial(123) + a\cdot \partial(234) + b\cdot
\partial(124)$, for any possible $a,b \in \R$.  It is easy to see that
for $a=-1, b=0$ one get $C_4 = \partial(1,2,3) - \partial(1,2,4)  =
-\partial(2,3,4) - \partial(1,3,4)$. Hence the right hand side forms a
filling of size $2$ which a $1$-deficit filling. For every other
setting of $a,b$ (that is not isomorphic) it is easy and left for the
reader to verify that $C$ is a sum of boundary of three simplices -
hence a $0$-deficit filling.

\item[$n=5$] 

  In this case the collection of cycles is much larger. 
We refer to two main cases according to whether there is a vertex
$v$ in the cycle with exactly two adjacent edges, or the case where
all vertices in the cycle are
adjacent to at least $3$ edges.

{\bf case1: } The first case in which there is a vertex $v$ of degree $2$
in $C$: we choose $v_5 =v$ and apply
$\FILL$.  Since $v$ has two adjacent edges, it follows that one is
incoming, the other is outgoing, both with the same weight which is
w.l.o.g. $1$. Assume these are the directed edges $(4,5), (5,1)$. Then
calling the $\FILL((\Link(v_5,C), ~ [4])$  would find $F_4^1$ which is a
weighted Hamiltonian path on $4$ vertices from $4$ to $1$
carrying a weight $1$ (and there are $2$ such different paths). Now $Z_5^1 - \Star(5,Z_5^1)$ is a flow network
carrying a total of $1$ flow from $1$ to $4$. Hence (by simple flow
argument) either this flow is along a simple path of length $3,2$ or $1$.
Namely $Z_5^1 = C_5,C_4$ or $C_3$. In the first two cases  $F_4^1$ can
be taken so not to cancel this
path which will result in the cycle $Z_4^1$ that is not $C_4$ and 
hence a $0$-deficit will be constructed for it and for $C$.

The  problematic case above is when the flow is along one path of
length $1$.  Namely $Z_5 = C_3$. In this case, of the two possible
$F^{(1)}_4$, one
results in an empty cycle and the other with $Z_4^1 = C_4$ which will result in a $1$-deficit filling.
Hence for $C_3$ we end up in a $1$-deficit filling. Moreover, this
is best possible as it can be seen that any acyclic set on $K_5^2$ is
a construction as described in
Claim~\ref{cl:cone_extension}. It follows that  if there were a $0$-deficit
filling for $C_3$ it would be also achieved by $\FILL$. 

We conclude that there is
no such filling for $C_3$. We also note that there are several $1$-deficit
fillings.

Another subcase is when the above $1$-flow
from $1$ to $4$ is not on a simple path. It then can be split into two or more distinct paths. In that case,
again, the resulting graph $Z_4^1$ can be made not to be $C_4$ resulting in a $0$-deficit.

{\bf case 2:} The other case that is left is where every vertex in $Z_5^1$ is
adjacent to at least $3$ edges. Assume first that there is a vertex
$v$ that is adjacent to exactly $3$ edges. Assuem w.l.o.g that $v=5$ and
choose $v_5=5$ in $\FILL$. Hence$F_4^1$ is a tree whose boundary is the
three neighbours of $v_5$, which are w.l.o.g. $\{1,2,3\}$. Then any
tree on $[4]$ with $4$ being non-leave vertex can be made to be a
suitable $F_4^1$. Each will result a different labeled $Z_4^1$.  Since
there are $7$ such trees, with only $6$ possible labeled $C_4$ at
least one will result $Z_4^1 \neq C_4$ and the case $n=4$ will
guarantee a $0$-deficit filling.  (We note that none of the possible
$F^{(1)}_4$ will result in an empty cycle on account that all vertices in
$Z_5^1$ have degree at least $3$).

Finally, in the case of every vertex in $Z_5^1$ of degree $4$ makes
$Z_5^1$ a weighted orientation of $K_5$. In that case $Z_4^0 =
\Link(5,Z_5^1)$ is a weight $\{a,b,c,d\}$ on $[4]$ with $a+b+c+d=0$
and $F_4^1$ is a
weighted tree whose net weight on $[4]$ is as above. If for no proper
set of $[4]$ the weights sum to $0$, it is easy to see that any tree
on $4$ vertices can be weighted be a $0$-filling of the weighted
$Z_4^0$ as above. There are $16$ such trees and only $7$ forbidden
configurations for $Z_4^1$ (the 6 labeled $C_4$ + the empty cycle). As each tree results in a different
labeled configuration, at least one will result in a good $Z_4  \neq
C_4$ for the next recursion level. 

If, on the other hand $Z_4^1$ is the weighting $(1,-1,a,-a)$ then it
can be seen that $F_4^1$ being any of the $4$ stars can be weighted to
be a filling of the above. More over, it can be seen that at least one
of these possibilities will either result in $Z_4^1$ not being
$C_4$. Again, by the case of $n=4$ this will result in a $0$-deficit
filling for $C$.
 
\item[$n=6$]

Analysis in the spirit of $n=5$ is simpler here. If there is vertex
$v$ in
$Z_6^1$ that is adjacent to two edges we chose $v_6=v$ in $\FILL$. 
Then w.l.o.g $G=Z_6^1 - \Star(6,Z_6^1)$ is a flow network
carrying a total of $1$ flow from $1$ to $5$. Hence $F_5^1$ must be a
Hamiltonian path from $5$ to $1$. Here if $|V(Z_6^1)| \leq 5$ it is
immediate that such path (in fact at least two paths) can be taken to result in nonempty $Z_5^1$
having a vertex of degree $3$ (or more) and hence not $C_3$. It can
also be verified that if $Z_6^1 = C_6$ the same can be forced as well.

If the flow network defined by $G$ is a union of two or more distinct
path, again, the same holds, by the freedom we have due to the
relatively large number of
Hamiltonian paths.

Otherwise, if every vertex in $Z_6^1$ is adjacent to at least $3$
edges, and there is a vertex adjacent to exactly $3$ edges $v$ we set
$v_6=v$ in $\FILL$.  Assume that $V\setminus \{v\} = [5]$ and that $v$
is adjacent to $1,2,3$, then any tree in which $4,5$ are not leaves
can serve as $F_5^1$ (with a corresponding uniquely define
weighting). As there are $30$ such labeled trees and only ${5 \choose
  3}=10$ labeled $C_3$ there is at least two trees that will produce
an non-empty $Z_5^1 \neq C_3$.

If all vertices in $Z_6^1$ have degree $4$ or more, then $Z_6^1$ has
at least $12$ edges and $G = Z_6^1 - \Star(6,Z_6^1)$ has at least $8$
edges (where $6$ is chosen to be the vertex of the smallest
degree). But $F_5^1$ which is a tree on $5$ vertices has $4$ edges,
hence added to $G$ will result in a graph with at least $4$ edges
which cannot then be $C_3$. This ends the proof for this case.
\end{description}

\section{Claims for the proofs of Theorem \ref{thm:d=3_f2}}\label{app:d=3}

All Claims here are w.r.t $1$-dim complexes, namely graphs. For a
graph $G$ we denote by $$Odd(G) = \{v \in V(G)| ~ ~ deg(v,G) \equiv 1 ~
(\bmod ~ 2)\} $$

We use the following simple Claim on filling for $d=0$.
\begin{claim}
  \label{cl:basic}
Let $O \subseteq V$ with $|O|\equiv 0 ~ (\bmod 2)$, $~w \in O$ and
$y \in V$. Then there is a
$0$-deficit filling of $O$, i.e., a tree $T$ on $V$ with $Odd(T)=O$ in
which $\Star(w,T) = (y,w)$, namely the only neighbour of $w$ in $T$ is $y$.
\end{claim}
\begin{proof}
If $O = \{w,y\}$ then any Hamiltonian path with ends $w,y$ is the
required $T$. Otherwise, define $O' = (O \setminus \{w\}) \xor \{y\}$,
and construct any tree $T'$ on $[n-1] \setminus \{u,u'\}$ with $Odd(T')
= O'$ (which is possible by constructing  $0$-deficit filling for $d=0$). Then add the edge
$(w,y)$ to $T'$ to obtain $T$. 
  \end{proof}
\noindent
{\bf Claims for Case 1}.

\begin{claim}
  \label{cl:case1}
Let $u,u' \in V(Z_{n-1}^1)$ such that $(u,u') \in Z_{n-1}^1$. Let
$v_{n-1}=u$ be the pivot vertex chosen in $\FILL(Z_{n-1}^1)$. There
are two distinct  $0$-deficit filling $F_{n-2}^1, \tilde{F}_{n-2}^1$
each being $\Fill(Z^1_{n-1} \setminus \{u\})$
such that $Z^1_{n-2} = (Z_{n-1}^1 \setminus \{u\}) \xor F_{n-2}^1$
contains $u'$ in its vertex set.
\end{claim}
\begin{proof}

Let $G = Z_{n-1}^1 \setminus \{u\}$ be the  graph on the
vertex set $[n-2]$. Then $u' \in O=Odd(G)$. 

If there is $y \notin  \{v_n,u,u'\}$ such that $(u',y)
\notin G$, then let $T=T_{n-2}$ be a tree on $[n-1] \setminus \{u\}$ as
asserted by Claim \ref{cl:basic} w.r.t $O$,  $w=u'$ and $y$.
 The resulting $Z_{n-2}^1$ that is defined by $F_{n-1}^1=T_{n-2}$ will
contain the edge $(u,y')$  and hence $u'$ as a vertex.  

The above does not happen only if in $G$, $u'$ is connected to all the other
$n-3$ vertices in $[n]\setminus \{v_n, u,u'\}$.  Since $n-3 \geq 2$ this
means that it has degree at least two in $G$. Using the same
$T_{n-2}$ as above will result in $u'$ being
in $Z_{n-2}^1$. This is true  as $u'$ has at least two edges in $G$
of which at most one can be canceled by the single edge containing
$u'$ in $T$.
\end{proof}

\newpage
\noindent
{\bf Claims for Case 2}.

\begin{claim}
  \label{cl:case_2.1}
Let $u,u' \in V(Z_{n-1}^1)$ such that $(u,u') \notin Z_{n-1}^1$. Let
$v_{n-1}=u$ be the pivot vertex chosen in $\FILL(Z_{n-1}^1)$. There is
a $0$-deficit filling $F_{n-2}^1 = \FILL(Z^1_{n-1} \setminus \{u\})$
such that $Z^1_{n-2} = (Z_{n-1}^1 \setminus \{u\}) \xor F_{n-2}^1$
contains $u'$ in its vertex set.
\end{claim}
\begin{proof}
    $F^{(1)}_{n-2}$ should be a tree $T_{n-2}$ on $[n-2] = V \setminus
  \{v_{n-1},u\}$ that has $O=Odd(T_{n-2}) = \Link(u,Z_{n-1}^1)$ and
  such that $u' \in V(Z_{n-2}^1)$ that is resulted by $T_{n-2}$. The
  construction of $T_{n-2}^1$ in this case is simple. Construct first
  $T'$ on $[n-1]\setminus \{u,u'\}$ with $Odd(T')= O$.  This is
  possible by the $0$-dim filling case. Then subdivide an edge $e=(x,y)$
  of $T'$ by replacing it with $(x,u'),(u',y)$. The resulting
  $T_{n-2}$ is a tree with $Odd(T_{n-2})=O$ regardless of the choice
  of $e$. Now, if $u'$ is adjacent to $4$ or more vertices in
  $Z_{n-1}^1$ then any choice of $e$ will result in $u' \in
  Z_{n-2}^1$. If $u'$ is adjacent to exactly two neighbours $a,b$ in
  $Z_{n-1}^1$ then $e$ can be any edge of $T'$ except for $(a,b)$.

Hence since $n-4 \geq 2$, $T'$ has at least two edges, one of them is
certainly a good choice for $e$.
 \end{proof}

\noindent
{\bf Claims for Case 3}.

\begin{claim}
  \label{cl:case_2_d=3} 
Assume that $Z_{n-1}^1$ is composed of a disjoint union of at most two
cliques, each  being monochromatic with respect to $A(*)$, and of
different values if it is not just one clique. Then $Z_{n-1}^2$ can be
made friendly. Moreover,  two corresponding $F_{n-1}^2$  of
$0/1$-deficit as needed can be  constructed.
\end{claim}

\begin{proof}
Suppose first that $Z_{n-1}^1 = K_{n-1}$, namely it is just one
clique, monochromatic with respect to $A(*)$.

Let $v = v_{n-1}$ and $G=K_{n-1} \setminus \{v\}$ the resulting
clique on $[n-2]$. Then with $T_{n-2}$ being a star centered at
$y$, the resulting $Z_{n-2}^1$ does not have $y$ in its vertex
set. Choosing any $u \neq y$ as $u=v_{n-2}$ as the pivot vertex in
$\FILL(Z_{n-2}^1)$ will result in a cycle
$Z_{n-3}^1$ in which $y \in V(Z_{n-3}^1)$.    We then choose
the next pivot $v_{n-3}=y$ for $\FILL(Z_{n-1}^1)$.

It follows (by Claim \ref{cl:friendly_internal}) that $deg(v_{n-1}, Z_{n-1}^2) = A(v) \xor B(v) \equiv A(v)
\xor (n-3) ~ (\bmod 2)$. On the other hand Claim
\ref{cl:friendly_internal} implies that   $deg(y,Z_{n-1}^2) = A(v)
\xor B(y) \equiv A(v) \xor deg(y,\FILL(Z_{n-2}^1)) \xor 1$ where the $1$ comes
from that fact that $y$ is a neighbour of $v$ in $Z_{n-1}^1$. 
Using Claim \ref{cl:friendly_internal} again (the part on $(B(u))$),
implies that $deg(y,\FILL(Z_{n-2}^1) =
deg(y,\FILL(Z_{n-3}^1) \xor 0$ where the $0$ comes from the fact that
$y$ is not a neighbour of $v_{n-2}$.  Finally, one last application of
Claim \ref{cl:friendly_internal} implies that $deg(y,\FILL(Z_{n-3}^1) =
n-5$. Substituting  we get $deg(y,Z_{n-1}^2) = A(v) \xor 1 \xor (n-5)
\equiv 1 + deg(v,Z_{n-1}^2)$ and we conclude that $Z_{n-1}^2$ is friendly. 
We note that in all the above we assume that $Z_{n-3}^1$ is non empty
which is true since $n \geq 6$.

Assume now that $Z_{n-1}^1$ is a disjoint union of two cliques, each
monochromatic w.r.t  $A(*)$. Let $K_{\ell}$ be the largest of these
cliques. The situation here is very similar to the previous case: we
set $v_{n-1}=v$ for an arbitrary vertex in $K_\ell$.  Let $y \in
K_\ell \setminus \{v_{n-1}\}$. Assume we can construct a $T=T_{n-2}$ with $O= Odd(T) =
  V(K_\ell) \setminus \{v_{n-1}\}$ for which  
  $(u,y) \notin Z_{n-2}^1, ~ u \in V(Z_{n-2})$ and $y \in K_{\ell}
  \setminus \{v,u\}$. Suppose further that choosing $v_{n-2}=u$
  results in $Z_{n-3}^1$ for which $y \in 
  V(Z_{n-3}^1)$. If such $T$ exists then we are exactly  in the
  situation of the 
  previous case (w.r.t $Z_{n-3}, ~y, v$) and choosing $v_{n-3}=y$ will end
  the proof as in that case.   

To construct $T'$ as needed, we take a star centered at $y$ with leaves
$O \setminus \{y\}$. We then subdivide an arbitrary edge of this star
$e=(y,a)$ by inserting all other vertices not in $K_\ell$. Namely, we
replace $e$ by a path from $y$ to $a$ containing all vertices not in
$K_\ell$. Note that $Odd(T)$ is as needed. Further, $Z_{n-2}$ will not
contain the edges $(y,x)$ for every $x \in K_\ell \setminus
\{v,y\}$. Hence choosing $u$ to be any of these vertices $x$ and
constructing $T_{n-3}$ as asserted in Claim \ref{cl:case_2.1} w.r.t
$u$ and $u'=y$ (and $n$ replaced by $n-1$) will result in $Z_{n-3}^1$
that contains $y$ in its vertex set.

We note that to apply Claim \ref{cl:case_2.1} we needed $n
\geq 6$ but since we have replaced $n$ with $n-1$ we get that $n \geq
7$ is needed, which is correct by our assumptions. 
 
Finally, the above implies a construction of one $F_{n-1}^2 =
\Fill(Z_{n-1}^1)$ that is $0/1$-deficit as needed. Since in both cases
$Z_{n-1}^1$ contains a clique of size at least $3$ (as $n \geq 7$),
any permutation of the choices of the vertices inside one clique to
play the role of $u,v,y$ above will create
an isomorphic distinct $\Fill(Z_{n-1}^2)$  (this is since, e.g., 
$\Fill(Z_{n-1}^1)$ cannot be invariant to all such permutation on
account of the average degree of a pair is less than one, hence some
pairs are non-existant while some pairs are,  in the $1$-skeleton of any
acyclic $\Fill(Z_{n-1}^1)$).  
\end{proof}

{\bf Theorem \ref{thm:d=3_f2} for $n \leq 6$}.

\begin{claim}
  \label{cl:thm:d=3_f2_n<6}
Theorem \ref{thm:d=3_f2} is true for $n \leq 6$.
\end{claim}
\begin{proof}
  We have checked the statement for $n \leq 7$  by running the program
  in Section \ref{sec:program} below. It could be run for each $n \leq 7$. However, for $n=4,5$ the
  situation is simple enough to also  verify manually as explained
  below.

  For $n=4$, which is the smallest $n$ for which a non-empty $2$-cycle
  exists, the unique such cycle is $Z_4^2 = \partial_4 \sigma$, where
  $\sigma = (1,2,3,4)$ is the  unique $3$-simplex. 
  Then $\sigma$  (of size $1$) is a $0$-deficit filling. 

For $n=5$ the possible non-empty $2$-cycles (up to isomorphism) are:
$\partial K_4^3$, $\partial (1,2,3,5) + \partial (2,3,4,5)$, and
$\partial(5,1,2,3) +\partial(5,1,2,4) + \partial (5,2,3,4)$.

The first is not friendly. It can be verified that the $1$-deficit
$F_4^2$ that is resulted by $\FILL(Z_5^2) \setminus \{4\}$ will result
in a non empty cycle $Z_4^2$. Hence by the case $n=4$, it will have
a $0$-deficit filling which will  result  in a $1$-deficit filling of
$Z_5^2$.
 For the 2nd case, (which is friendly), one should, not apply $\FILL$
 with the 'good' vertex, as this results in an {\em empty}
 $Z_4^2$. However, if one chooses the bed vertex (e.g., $v_5=5$) one
 gets $Z_4^1 = C_4$ which has two distinct $1$-deficit fillings, one
 resulting in a non-empty $Z_4^2$ which by the case $n=4$ has  a
 $0$-deficit filling. Altogether, this gives a $1$-deficit filling for
 $Z_5^2$.

The same reasoning applies to the last case.

\subsection{program for small $n$'s}\label{sec:program}
We have checked the case of $\F_2, d=3$ and  small $n$ (for $n\leq 8$), by
a $C^++$ program that is available on the 2nd authors cite:

\begin{verbatim} 
http://cs.haifa.ac.il/~ilan/online-papers/online-papers.html/fillings.cpp
\end{verbatim}

The program (exponential in $n$) runs over all possible acyclic sets
on $K_n^3$, and for each it computes the boundary (which is a
$2$-cycle). In doing so, it also register for each $2$-cycle how many
times  it was
found as a boundary of a $0$-deficit or a $1$-deficit tree. 

For $n=7$ every cycle was found to be a boundary of at least two
$1$-deficit or $0$-deficit trees. For $n=8$ all cycles are boundaries
of at least two distinct $0$-deficit trees.
\end{proof}
\end{document}